\documentclass[reqno]{amsart}
\usepackage{hyperref}
\usepackage{amsmath}  
\usepackage{amssymb} 
\usepackage{mathrsfs} 
\usepackage{accents}
\usepackage{graphicx}

\usepackage{diagbox}

\usepackage{tikz} 
\usetikzlibrary{calc} 
\usepackage{xcolor}  
\usetikzlibrary{arrows.meta} 
\usepackage{amsthm} 
\usepackage{float}
\usepackage{enumitem}
\usepackage[english]{babel}
\usepackage[font=footnotesize, labelfont=bf]{ caption}
\usepackage{ esint }
\usepackage{ dsfont }

\newcommand{\eg}{{\it e.g.}}
\newcommand{\ie}{{\it i.e.}}

\theoremstyle{plain}
\begingroup
\newtheorem{theorem}{Theorem}[section]
\newtheorem{lemma}[theorem]{Lemma}
\newtheorem{proposition}[theorem]{Proposition}

\endgroup

\theoremstyle{definition}
\begingroup

\AtBeginEnvironment{definition}{%
  \pushQED{\qed}%
}
\AtEndEnvironment{definition}{\popQED\endexample}
\newtheorem{example}[theorem]{Example}
\endgroup
\theoremstyle{remark}
\newtheorem{remark}[theorem]{Remark}
\AtBeginEnvironment{remark}{%
  \pushQED{\qed}%
}
\AtEndEnvironment{remark}{\popQED\endexample}
\renewcommand{\tilde}{\widetilde}
\newcommand{\sm}{\setminus}

\renewcommand{\d}{ \mathrm{d}}

\newcommand{\x}{{\times}}
\newcommand{\ev}{\mathrm{ev}}

\newcommand{\diag}{\mathrm{diag}}
\renewcommand{\div}{\mathrm{div}}

\numberwithin{equation}{section}
\newcommand{\A}{\mathcal{A}}

\newcommand{\R}{\mathbb{R}}
 
\renewcommand{\P}{\mathbb{P}} 
\newcommand{\Prob}{\mathcal{P}}

\newcommand{\de}{\partial} 
\newcommand{\e}{\varepsilon}
\newcommand{\U}{\mathcal{U}}
\newcommand{\F}{\mathscr{F}}
\newcommand{\W}{\mathcal{W}}
\newcommand{\M}{\mathcal{M}}
\newcommand{\E}{\mathbb{E}}
\newcommand{\D}{\mathrm{D}}
\renewcommand{\O}{\mathcal{O}}

\usepackage{mathtools}
\mathtoolsset{showonlyrefs}  

\newcommand{\myequation}{\begin{equation}}
  \newcommand{\myendequation}{\end{equation}}
  \let\[\myequation
  \let\]\myendequation

  \title[Replicator dynamics as the large population limit of a discrete Moran process]{Replicator dynamics as the large population limit of a discrete Moran process in the weak selection regime: A proof via Eulerian specification}
  
  \author[M. Morandotti]{M. Morandotti}
  \address[Marco Morandotti]{\newline Dipartimento di Scienze Matematiche ``G. L. Lagrange'', Politecnico di Torino, Corso Duca degli Abruzzi 24, I--10129 Torino, Italy.}
  \email[]{marco.morandotti@polito.it}
  
  \author[G. Orlando]{G. Orlando}
  \address[Gianluca Orlando]{\newline
   Dipartimento di Meccanica, Matematica e Management, Politecnico di Bari,
   Via E.~Orabona 4, I--70125 Bari, Italy.}
  \email[]{gianluca.orlando@poliba.it}


\begin{document}

\maketitle

\begin{abstract}
  We study the large population limit of a multi-strategy discrete-time Moran process in the weak selection regime. 
  We show that the replicator dynamics is interpreted as the large-population limit of the Moran process.
  This result is obtained by interpreting the discrete process in its Eulerian specification, proving a compactness result in the Wasserstein space of probability measures for the law of the proportions of strategies, and passing to the limit in the continuity equation that describes the evolution of the proportions.
\end{abstract}

\tableofcontents

\section{Introduction}

Understanding the behavior of large populations of interacting agents is crucial in a variety of scientific fields, including biology~\cite{SmiPri73,TayJon78}, sociology \cite{DMPW09,Tos06}, ecology \cite{DorBlu05,Ken10}, artificial intelligence \cite{Hol75}, and economics \cite{Car21}.  
Evolutionary game theory provides the mathematical framework to model and analyze these dynamics, in which agents are typically endowed with strategies. These are rationally selected by each agent according to complex mechanisms, often influenced by the interplay of the individuals among themselves and with the environment, and based on an underlying optimality principle. 

In this context, the replicator dynamics accomplishes the task of modeling the evolution of strategies in the population by subjecting the probability of reproduction to their fitness: The higher the fitness, the higher the chance of being selected to reproduce. 
More specifically, letting $\U = \{u_1, \dots, u_M\}$ be the set of strategies and $\lambda(t) = (\lambda_{u_1}(t), \dots, \lambda_{u_M}(t))^\top$ be the vector collecting the proportions of agents with strategies $u_1,\dots,u_M$ at time $t$, and given a fixed payoff matrix $A = (a_{ij}) \in \R^{M \x M}$, the fitness of strategy $u_i$ at time $t$ is modeled as (see~\cite[Chapter~8]{HofSig98})
\[
(A \lambda(t))_i - (A \lambda(t)) \cdot \lambda(t) = \sum_{j=1}^M a_{ij} \lambda_{u_j}(t) - \sum_{\ell=1}^M\sum_{j=1}^M a_{\ell j} \lambda_{u_j}(t) \lambda_{u_\ell}(t)  \, , \quad \text{for } i = 1, \dots, M \, .
\]
Interpreting $a_{ij}$ as the payoff of an agent with strategy $u_i$ interacting with an agent with strategy~$u_j$, the formula above measures the performance of the selected strategy $u_i$ compared to the average performance of all the strategies.
The replicator dynamics encodes the Darwinian principle that if~$u_i$ outperforms the average, then its fitness is positive and its spreading within the population is favored. Vice versa, strategies which underperform are progressively suppressed.
In formulae, the replicator dynamics evolves according to the system of ODEs
\[ \label{eqintro:replicator_dynamics}
\begin{split}
  \frac{\d}{\d t} \lambda_{u_i}(t) &= \lambda_{u_i}(t) \big( (A \lambda(t))_i - (A \lambda(t))\cdot \lambda(t) \big) \\
  & \eqqcolon b_i(\lambda(t)) \, ,
\end{split} \quad \quad  t \in (0,T) \, ,  \quad \text{for } i = 1, \dots, M \, ,
\]
subject to a given initial condition $\lambda(0) = \lambda^0$. 
The criterion followed by individual agents to select their strategies is latent in~\eqref{eqintro:replicator_dynamics}, which describes the continuous-time evolution of the proportions of strategies in an averaged fashion.

The objective of this paper is to provide the mathematical framework to derive~\eqref{eqintro:replicator_dynamics} as a mean-field limit of a discrete stochastic process modeling this evolutionary mechanism from the point of view of individual agents and their pairwise interactions.

\begin{figure}[h]
  \includegraphics[scale=0.8]{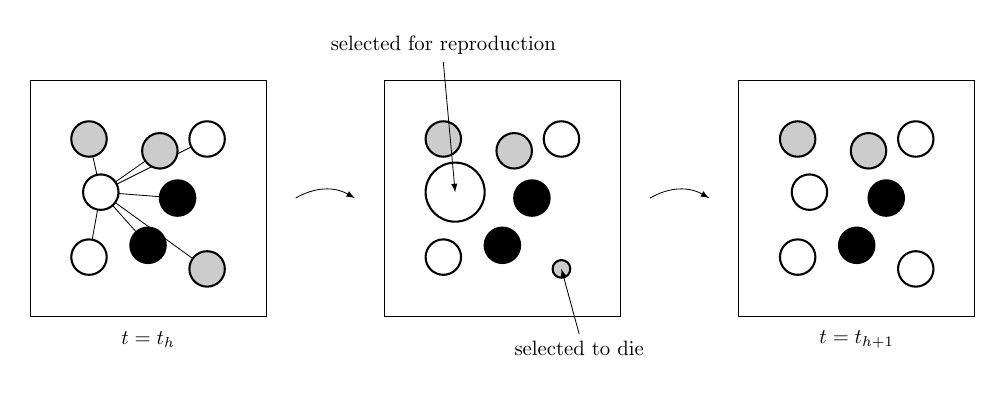}
  \caption{Graphical representation of one time step of a Moran process. In the picture, we have $N = 8$ agents and $M = 3$ strategies, $u_1$ (white circle), $u_2$ (grey circle), $u_3$ (black circle). At time $t_h$, for each agent a fitness is calculated in terms of the expected payoff due to the interaction with an agent sampled randomly uniformly in the population. The fitness is used to determine the probability of reproduction. In the figure, an agent with strategy $u_1$ has been chosen for reproduction (biggest circle), and an agent with strategy $u_2$ has been chosen to die (smallest circle). The generation is updated at time $t_{h+1}$ accordingly.}
  \label{fig:Moran_process}
\end{figure}

Before describing in detail this process, we recall the classical \emph{Moran process}~\cite{Mor62}, a prototypical example of discrete stochastic process used in evolutionary biology to model natural selection in finite populations with two strategies (alleles, in biology) competing for dominance. 
Letting $M = 2$, the set of strategies becomes $\U = \{u_1, u_2\}$. 
We denote by $N$ the number of individuals in the population, and by $N_{u_i}$ the number of individuals with strategy $u_i$ at a given time.
If no fitness affects the evolution (\ie, a neutral mutation scenario), the Moran process is a discrete-time Markov chain for $N_{u_1}$ with uniform birth-death transition probabilities given by: 
\[ \label{eqintro:vanilla_Moran_process}
\begin{split}
  & \P\Big(\text{``$N_{u_1}$ increases by 1''}\Big) = \frac{N_{u_1}}{N} \frac{N_{u_2}}{N}  = \frac{N_{u_1}}{N} \frac{N-N_{u_1}}{N} \\
  & = \P\Big(\text{``a $u_1$-individual is chosen for reproduction''}\Big) \P\Big(\text{``a $u_2$-individual is chosen to die''} \Big) \, ,\\
  & \P\Big(\text{``$N_{u_1}$ decreases by 1''}\Big) = \frac{N_{u_2}}{N} \frac{N_{u_1}}{N} = \frac{N-N_{u_1}}{N} \frac{N_{u_1}}{N} \\
  & = \P\Big(\text{``a $u_2$-individual is chosen for reproduction''}\Big) \P\Big(\text{``a $u_1$-individual is chosen to die''} \Big) \, ;
\end{split}
\]
by complementarity, one obtains the probability that $N_{u_1}$ remains unchanged. 

When fitness of strategies is introduced, the Moran process is modified to account for the fact that the higher the fitness of a strategy, the higher the probability that one of its bearers is selected to reproduce.
More precisely, letting $f_{u_i}$ be the fitness of individuals with strategy $u_i$, the reproduction probabilities in~\eqref{eqintro:vanilla_Moran_process} are replaced by
\[ \label{eqintro:fitness_Moran_process_2strategies}
\begin{split}
  \P\Big(\text{``a $u_1$-individual is chosen for reproduction''}\Big) & = \frac{N_{u_1}f_{u_1}}{N_{u_1}f_{u_1} + N_{u_2} f_{u_2}} \, , \\
  \P\Big(\text{``a $u_2$-individual is chosen for reproduction''}\Big) & =  \frac{N_{u_2}f_{u_2}}{N_{u_1}f_{u_1} + N_{u_2} f_{u_2}}  \, .
\end{split}
\]
Let us discuss the case where a payoff matrix $A \in \R^{2\times 2}$ is used to describe interactions. An individual carrying strategy $u_1$ interacts with another individual chosen randomly uniformly among the others in the population and has an expected payoff 
\[
\pi_{u_1} = a_{11} \frac{N_{u_1}-1}{N-1} + a_{12} \frac{N_{u_2}}{N-1} \, .
\]
Analogously, for an individual carrying strategy $u_2$, the expected payoff is
\[
\pi_{u_2} = a_{21} \frac{N_{u_1}}{N-1} + a_{22} \frac{N_{u_2}-1}{N-1} \, .
\]
A common choice~\cite{Now06} is to define the fitness of individuals with strategy $u_i$ as a convex combination of the expected payoff and the constant $1$, \ie,
\[
f_{u_i} = (1-w) + w \pi_{u_i} \, , \quad \text{for } i = 1,2 \, ,
\]
where $w \in [0,1]$ is a weight that interpolates between the case of neutral selection $w = 0$, giving~\eqref{eqintro:vanilla_Moran_process}, and the case of strong selection, $w = 1$, where the fitness is solely determined by the payoffs.
The regime $w \approx 0$ is typically referred to as the \emph{weak selection} regime, where the payoffs provide a small contribution to the fitness.

In this paper, we consider a generalization of the Moran process to a population with strategies $\U = \{u_1, \dots, u_M\}$, $M \geq 2$, see Figure~\ref{fig:Moran_process}.
In this case, the probability that an individual with strategy $u_i$ is chosen for reproduction (as in~\eqref{eqintro:fitness_Moran_process_2strategies}) is given by
\[ \label{eqintro:probability_Moran_process_Mstrategies}
\begin{split}
  \P\Big(\text{``a $u_i$-individual is chosen for reproduction''}\Big) & = \frac{N_{u_i}f_{u_i}}{\sum_{j=1}^M N_{u_j}f_{u_j} } \, ,
\end{split}
\]
where $N_{u_i}$ is the number of individuals with strategy $u_i$ in the population and $f_{u_i}$ is the fitness of individuals with strategy $u_i$ defined via the convex combination
\mathtoolsset{showonlyrefs=false}
\[ \label{eqintro:fitness_Moran_process_Mstrategies}
f_{u_i} = (1-w) + w \pi_{u_i} \, , \quad \text{for } i = 1, \dots, M \, .
\]
\mathtoolsset{showonlyrefs=true}
The expected payoff of an individual with strategy $u_i$ interacting with another individual chosen uniformly at random in the population is now given by
\[ \label{eqintro:expected_payoff_Mstrategies}
\pi_{u_i} = \sum_{\substack{j=1 \\ j \neq i}}^M a_{ij} \frac{N_{u_j}}{N-1} + a_{ii} \frac{N_{u_i}-1}{N-1} \, , \quad \text{for } i = 1, \dots, M \, .
\]
In this paper we study rigorously the asymptotic behavior of this generalized Moran process by considering:
\begin{itemize}
  \item the large population limit $N \to +\infty$;
  \item the discrete-time to continuous-time limit; letting $\tau$ be the time step of the process, this corresponds to $\tau \to 0$;
  \item the weak selection regime $w \to 0$.
\end{itemize}
From our analysis it emerges that the relation between the three limits have to be taken simultaneously and with a precise relation among the parameters to obtain a nontrivial and meaningful result.
 To describe this relation between these parameters, we introduce sequences $N_k \to + \infty$, $\tau_k \to 0$, and $w_k \to 0$.

Our main result can be summarized as follows:
Assuming that 
\[ \label{eqintro:scaling_parameters}
N_k \sim \tau_k^{-\alpha} \, , \quad  w_k \sim \tau_k^\beta \, , \quad \text{for some } \alpha, \beta > 0 \, , \quad \alpha + \beta = 1 \, , \quad  \alpha > \frac{1}{2} \, ,
\]
the replicator dynamics describes the limit as $k \to \infty$ of the generalized Moran process modeled by~\eqref{eqintro:probability_Moran_process_Mstrategies}--\eqref{eqintro:expected_payoff_Mstrategies}.

Previous works in the literature addressed the problem of deriving the replicator dynamics from discrete stochastic processes, see, e.g., \cite{BorSar97, TraClaHau05, TraClaHau06, ChaSou09, Hil11, ChaSou14}. 
The results in~\cite{TraClaHau05, TraClaHau06, ChaSou09, ChaSou14} are aligned with the scaling assumptions~\eqref{eqintro:scaling_parameters}. 
In fact, they are not merely technical, as they determine the PDE governing the limit dynamics. 
In particular, if $\alpha + \beta > 1$, then the result is trivial (the evolution is constant). If $\alpha + \beta < 1$, a different time scaling would be necessary. Finally, if $\alpha = 1/2$, then higher order terms have to be considered in the limit dynamics~\cite{TraClaHau05, TraClaHau06, ChaSou09, ChaSou14}.

We describe now more in detail the results obtained in this paper.
The discrete process is fully described by the stochastic evolution of the vector $\lambda^k(t_h)$ collecting the proportions of strategies at discrete times $t_h = h \tau_k$, for $h = 0, \dots, k$, see Section~\ref{sec:discrete_stochastic_process}. 
Noticing that the components of $\lambda^k(t_h)$ sum to one, we can interpret $\lambda^k(t_h)$ as a random point in the $(M-1)$-dimensional simplex $\Delta^{M-1} \subset \R^M$ (recall that $M$ is number of available strategies). 
The discrete path $\lambda^k(t_h)$ is suitably interpolated to obtain a continuous path $\lambda^k \colon [0,T] \to \Delta^{M-1}$, see~\eqref{eq:piecewise_affine_interpolation}. 
The asymptotic behavior of the discrete process as $k \to \infty$ is obtained in terms of the limit of the law $\Lambda^k_t \in \Prob(\Delta^{M-1})$ of the random vector $\lambda^k(t)$ (see~\eqref{eq:Lambda_k_t}); for every $k$, the map $t \mapsto \Lambda^k_t \in \Prob(\Delta^{M-1})$ is a continuous path in the space of probability measures on the simplex $\Delta^{M-1}$.
In Proposition~\ref{prop:approximate_distributional_solution}, we show that this path solves, in the sense of distributions, 
\[ \label{eqintro:approximate_limit}
\de_t \Lambda^k_t + \div \big(  b  \Lambda^k_t \big) \approx 0 \, , \quad \text{(with initial condition),}
\]
a continuity equation driven by the velocity field $b = (b_1, \dots, b_M)^\top$ with components defined in~\eqref{eqintro:replicator_dynamics}.
We turn the reader's attention to the symbol $\approx$ in~\eqref{eqintro:approximate_limit}, meaning that the equation holds in an approximate sense, up to an error term vanishing as $k \to \infty$ under the assumptions~\eqref{eqintro:scaling_parameters}.

A crucial technical step in our result is showing compactness of the sequence of paths $t \mapsto \Lambda^k_t$, that guarantees a meaningful limit as $k \to \infty$ to be taken in~\eqref{eqintro:approximate_limit}.
Since the probability measures $\Lambda^k_t$ have finite first moment (see Remark~\ref{rem:compactness_preliminaries}), it is natural to endow $\Prob(\Delta^{M-1})$ with the $1$-Wasserstein distance $\W_1$.
As we illustrate in Remark~\ref{rem:nontrivial_compactness}, straightforward estimates for the discrete stochastic process seem not to guarantee that the paths $t \mapsto \Lambda^k_t$ satisfy equi-continuity properties in the Wasserstein distance, ruling out an immediate application of available compactness results for paths of probability measures.
For this reason, we prove compactness in Theorem~\ref{thm:properties_of_Lambda} by deriving a more refined estimate and exploiting it in an argument that follows the lines of the proof of the original Arzel\`a--Ascoli theorem.
Once compactness is established, we can pass to the limit in~\eqref{eqintro:approximate_limit} and obtain that $\W_1(\Lambda^k_t, \Lambda_t) \to 0$ as $k \to \infty$ for every $t \in [0,T]$, where $\Lambda_t \in \Prob(\Delta^{M-1})$ is the solution to 
\[ \label{eqintro:limit}
\de_t \Lambda_t + \div \big(  b  \Lambda_t \big) = 0 \, , \quad \text{(with initial condition)}.
\]
We recognize that the latter equation is the Eulerian specification of the replicator dynamics~\eqref{eqintro:replicator_dynamics}, see Section~\ref{sec:replicator_dynamics}.

We conclude this introduction by remarking that the technique in this paper allows us to obtain the limit result with no regularity ansatz on the distribution of proportions $\Lambda^k_t$.
The approach we follow is inspired by \cite{AlmMorSol21,AmbForMorSav21,MorSol20}, where the equivalence of the Eulerian and Lagrangian notions of solution to~\eqref{eqintro:limit} is obtained via the \emph{superposition principle} \cite[Section~5.2]{AmbForMorSav21} (see also \cite[Theorem~8.2.1]{AmbGigSav08} and \cite{AmbTre14}).

\section{Notation and basic definitions}

In this section we collect some basic definitions and notation that will be used throughout the paper.

\subsubsection*{Probability}
Hereafter, we fix a probability space $(\Omega,\F,\P)$. 

If $X$ is a random variable, we usually adopt the notation $\hat X$ to indicate a fixed realization of the random variable $X$.

The expectation of a random variable $X$ is denoted by $\E[X]$.

We denote by $X_\# \P$ the pushforward of the probability measure $\P$ through the random variable~$X$, which defines a probability measure on the range of $X$.

\subsubsection*{$1$-Wasserstein distance} 
Let $(\M,d)$ be a complete, separable metric space, equipped with its Borel $\sigma$-algebra.  
We let $\Prob(\M)$ denote the set of Borel probability measures on $\M$.
The $1$-Wasserstein distance between two probability measures $\mu, \nu \in \Prob(\M)$ is defined by
\[
\W_1(\mu,\nu) = \inf_{\gamma} \int_{\M \x \M} d(x,y) \, \d \gamma(x,y) \, ,
\]
where the infimum is taken over all probability measures $\gamma \in \Prob(\M \x \M)$ with marginals $\mu$ and $\nu$.

We consider the subset of $\Prob(\M)$ consisting of probability measures with finite first moment, \ie,
\[
\Prob_1(\M) := \Big\{ \mu \in \Prob(\M) \ : \ \int_\M d(x_0,x) \, \d \mu(x) < +\infty \Big\} \, ,
\]
for some (and hence any) $x_0 \in \M$.
The $1$-Wasserstein distance is a metric on $\Prob_1(\M)$.

We recall the Kantorovich's duality formula for the $1$-Wasserstein distance~\cite[Remark~6.5]{Vil08}:
\[
\W_1(\mu,\nu) = \sup \Big\{ \int_\M \psi\, \d \nu - \int_\M \psi \, \d \mu \ : \   \psi \text{ 1-Lipschitz} \Big\} \, .
\]

Finally, we recall the following fact: If $\W_1(\mu_n,\mu) \to 0$ as $n \to \infty$, then 
\[
\int_\M \psi \, \d \mu_n \to \int_\M \psi \, \d \mu \, ,
\]
for every $\psi \in C(\M)$ such that $|\psi(x)|  \leq C(1 + d(x,x_0))$ for some $C > 0$ and some $x_0 \in \M$.



\section{The discrete stochastic process} \label{sec:discrete_stochastic_process}

In this section we describe in detail the multi-strategy Moran process studied in this paper.
Each subsection is dedicated to a specific aspect of the process.

\subsubsection*{Time}
We fix an equi-spaced discretization of step $\tau_k$ of a time interval $[0,T]$ given by
\[
0 = t_0 < t_1 < \dots < t_k = T \, , \quad  \tau_k := \frac{T}{k} \, , \quad t_{h} = h \tau_k  \, , \quad \text{for } h = 0, \dots, k \, .
\]

\subsubsection*{Population set}
We consider a population of $N\geq 2$ agents and we identify agents with elements $n \in \A := \{1,\dots,N\}$. 
We will be interested in the limit behavior of the population size $N \to \infty$. 
Later on, we will consider a sequence of population sizes of the form
\[ 
N_k \sim \tau_k^{-\alpha} \, , \quad \text{for some } \alpha > 0 \, .
\]
In this section, we stick to the notation $N$ for the population size, not to overload the notation.

\subsubsection*{Strategies}
Each agent is allowed to choose a strategy in a set $\U = \{u_1,\dots, u_M\}$.\footnote{When $M = 2$, the strategies $u_1$ and $u_2$ are typically called the \emph{mutant} and the \emph{resident}.} The strategy chosen by agent $n \in \A$ at the discrete time $t_h$, for $h \in \{0,\dots, k\}$, is represented by a random variable $S_n(t_h) \colon \Omega \to \U$. At each discrete time $t_h$, the state of the system is described by the random $N$-tuple 
\[
S(t_h) = (S_1(t_h) , \dots, S_N(t_h)) \colon \Omega \to \U^N \, .  
\]
For $i \in \{1,\dots,M\}$, we let $\A_{u_i}(t_h) \colon \Omega \to 2^{\A}$ be the random subset of agents with strategy $u_i$ at discrete time $t_h$ given by 
\[
\A_{u_i}(t_h) := \{n \in \A \ : \ S_n(t_h) = u_i \}   \, .
\]
Accordingly, we let $N_{u_i}(t_h) := \# \A_{u_i}(t_h)  \colon \Omega \to \{0,1,\dots,N\}$ be the random variable given by the number of agents with strategy $u_i$ at discrete time $t_h$. We remark that
\[ \label{eq:sum is N}
\sum_{i=1}^M N_{u_i}(t_h) = N \, .  
\]
Accordingly, for $i \in \{1,\dots,M\}$, we let $\lambda_{u_i}(t_h) := \frac{N_{u_i}(t_h)}{N}  \colon \Omega \to \{0,\frac{1}{N},\dots,\frac{N-1}{N},1\}$ be the proportion of agents with strategy $u_i$. We introduce the random vector collecting all proportions: 
\[ \label{eq:proportions}
\lambda(t_h) := (\lambda_{u_1}(t_h), \dots, \lambda_{u_M}(t_h) )^\top \colon \Omega \to [0,1]^M \, .
\]
We will refer to $\lambda(t_h)$ as the \emph{proportions} at time $t_h$. In fact, by~\eqref{eq:sum is N}, $\lambda(t_h)$ takes values in the $(M-1)$-dimensional simplex 
\[
\Delta^{M-1} = \Big\{ \lambda = (\lambda_1, \dots, \lambda_M)^\top \in [0,1]^M \ : \ \sum_{i=1}^M \lambda_i = 1 \Big\} \, .
\] 
 It follows that $\lambda(t_h)$ can be identified with the empirical probability measure on the strategy set $\mu(t_h) := \sum_{i=1}^M \lambda_{u_i}(t_h) \delta_{u_i}  \in \Prob(\U)$.

Later on, we will consider a sequence $N_k \to +\infty$ and the the proportions $\lambda(t_h)$ will depend on $k$ and will be denoted by $\lambda^k(t_h)$. 
In this section, we stick to the notation $\lambda(t_h)$, not to overload the notation.

If not explicitly specified, when a realization $S(t_h) = \hat S$ of the random state of the system $S(t_h)$ is observed, we will write $\A_{u_i}(t_h) = \hat \A_{u_i}$, $N_{u_i}(t_h) = \hat N_{u_i}$, and $\lambda(t_h) = \hat \lambda$ for the corresponding realizations of the random strategy sets, random number of agents, and random proportions, respectively.

\subsubsection*{Initial conditions} The process starts with an initial random state $S(0) = (S_1(0), \dots, S_N(0)) \colon \Omega \to \U^N$. 
This fixes an initial proportion $\lambda(0) = (\lambda_{u_1}(0), \dots, \lambda_{u_M}(0))^\top \in \Delta^{M-1}$.
We assume that the initial proportions $\lambda(0)$ are distributed according to a given law $\Lambda_0 \in \Prob(\Delta^{M-1})$.
This means that 
\[
\Lambda_0 = \lambda(0)_\# \P  \in \Prob(\Delta^{M-1}) \, .
\]
Later on, we will consider a sequence of initial laws $\Lambda_0^k \in \Prob(\Delta^{M-1})$.

\subsubsection*{Payoff matrix} The evolution of the strategies in the population will be described in terms of a payoff matrix $A = (a_{ij}) \in \R^{M \x M}$, where $a_{ij}$ is the payoff of an agent with strategy $u_i$ when interacting with an agent with strategy $u_j$. 
We shall assume that the entries of $A$ are non-negative, i.e., $a_{ij} \geq 0$ for all $i,j \in \{1,\dots,M\}$.
The payoff matrix is represented in the table below. 
\begin{table}[H]
  \centering
  \begin{tabular}{| c | c  c  c |}
    \hline
    \diagbox{1}{2} & $u_1$     & $\dots$  & $u_M$    \\
    \hline
    $u_1$          & $a_{11}$  & $\dots$  & $a_{1M}$  \\
    $u_2$          & $a_{21}$  & $\dots$  & $a_{2M}$  \\
    $\vdots$       & $\vdots$  & $\ddots$ & $\vdots$  \\
    $u_M$          & $a_{M1}$  & $\dots$  & $a_{MM}$  \\
    \hline
  \end{tabular} 
\end{table} 
At each time step $t_h$, an agent $n \in \A$ with strategy $u_i \in \U$ interacts with a different agent $n' \neq n$ chosen uniformly at random in $\A \sm \{n\}$. 
Specifically, given that the observed strategies at time $t_h$ are $S(t_h) = \hat S$, there are two possibilities for the second agent:
\begin{itemize}
  \item for $j \neq i$, agent $n'$ has strategy $u_j$ with probability $\frac{\hat N_{u_j}}{N-1} = \frac{N}{N-1} \hat \lambda_{u_j}$;
  \item for $j = i$, agent $n'$ has strategy $u_i$ with probability $\frac{\hat N_{u_i} - 1}{N-1} = \frac{N}{N-1}  \hat \lambda_{u_i}  - \frac{1}{N-1}$.
\end{itemize}
As a consequence, agent $n$ has an average payoff (depending on on the strategy $u_i$ of agent $n$, the proportions $\hat \lambda$, and $N$) given by
\[ \label{eq:average_payoff}
\pi_{u_i}(\hat \lambda) = \frac{N}{N-1} \sum_{j=1}^M a_{ij} \hat \lambda_{u_j} - \frac{1}{N-1} a_{ii} = \frac{N}{N-1} (A \hat \lambda)_i - \frac{1}{N-1} a_{ii} \, ,
\]
where $(A \hat{\lambda})_i$ denotes the $i$-th component of the matrix-vector product $A \hat \lambda$.

\subsubsection*{Fitness} Let $w \in [0,1]$ be a weight, whose precise value will be specified later. 
Given an agent $n \in \A$ with strategy $u_i \in \U$, and given that the observed strategies are $S(t_h) = \hat S$, we define the \emph{fitness} of agent $n$ at time $t_h$ as the convex combination
\[ \label{eq:fitness}
f_{u_i}(\hat \lambda) = (1-w) + w \pi_{u_i}(\hat \lambda)  \, .
\]
The weight $w$ is a parameter that determines the importance of the payoff in the fitness, as it interpolates between two extreme cases:
\begin{itemize}
  \item For $w = 0$, the fitness is constant and equal to $1$: This corresponds to a population where the agents are indifferent to the payoffs;
  \item For $w = 1$, the fitness is equal to the payoff. 
\end{itemize}
The regime $w \to 0$ is typically referred to as the \emph{weak selection} regime. 
We are interested in this regime, hence, in the rest of the paper, we will consider a sequence of weights $w_k \to 0$ of the form
\[ \label{eq:weight_sequence}
w_k \sim \tau_k^\beta \, , \quad \text{for some } \beta > 0 \, .
\]

Note that for $w_k$ small enough, the fitness $f_{u_i}(\hat \lambda)$ is non-negative.

\subsubsection*{Replication} At each new time step $t_{h+1}$, an agent $n \in \A$ is chosen to reproduce its strategy, depending on the state of the system at the time $t_{h}$.

The probability that an agent $n \in \A$ with strategy $u_i$ is chosen to reproduce its strategy depends on its fitness. 
We explain this precisely in the following. Let $R(t_{h+1}) \colon \Omega \to \A$ denote the random variable representing the agent chosen to reproduce its strategy at time $t_{h+1}$. 
Let $S(t_h) = \hat S$ be an observed realization of the random state of the system $S(t_h)$ at the previous time step.  
Then 
\[
\P\big(R(t_{h+1}) = n \ \big|\  S(t_h) = \hat S \big) = c f_{u_i}(\hat \lambda)  \quad \text{for every } n \in \hat \A_{u_i} \, .
\]
The value $c$ (depending on the status of the system) is chosen so that
\[
\sum_{\ell=1}^M \sum_{n \in \hat \A_{u_\ell}} \P\big(R(t_{h+1}) = n \ \big|\  S(t_h) = \hat S \big) = 1 \, ,
\]
hence 
\[
  1 = \sum_{\ell=1}^M \sum_{n \in \A_{u_\ell}(t_h)} c f_{u_\ell}(\hat \lambda) = c \sum_{\ell=1}^M \# \hat \A_{u_\ell}  f_{u_\ell}(\hat \lambda) =  c N \sum_{\ell=1}^M   \hat \lambda_{u_\ell}  f_{u_\ell}(\hat \lambda) \,.
\]
It follows that 
\[ \label{eq:probability_replication_u_i}
\begin{split}
  \P\big(R(t_{h+1}) \in \hat \A_{u_i} \ \big|\  S(t_h)  = \hat S \big) & = \sum_{n \in \A_{u_i}} \P\big(R(t_{h+1}) = n \ \big|\  S(t_h) = \hat S \big) \\
  & = N \hat \lambda_{u_i} \frac{1}{N \sum_{\ell=1}^M   \hat \lambda_{u_\ell}  f_{u_\ell}(\hat \lambda)} f_{u_i}(\hat \lambda)   =   \frac{\hat \lambda_{u_i} f_{u_i}(\hat \lambda)}{\sum_{\ell=1}^M   \hat \lambda_{u_\ell}  f_{u_\ell}(\hat \lambda)}  \,.
\end{split}
\]

Introducing the notation 
\[ \label{eq:average_fitness}
\overline f(\hat \lambda) = \sum_{\ell=1}^M \hat \lambda_{u_\ell} f_{u_\ell}(\hat \lambda) \, ,
\] 
to denote the average fitness of the population with proportions $\hat \lambda$, we conclude that 
\[
  \P\big(R(t_{h+1}) = n \ \big|\  S(t_h) = \hat S \big) = \frac{1}{\overline f(\hat \lambda)} f_{u_i}(\hat \lambda)  \quad \text{for every } n \in \hat \A_{u_i} \, .
\]

\subsubsection*{Abandoned strategy} 
At each new time step $t_{h+1}$, after the selection $R(t_{h+1}) = n$, an agent $n' \in \A$ (not necessarily distinct from the agent $n$) is chosen uniformly at random in $\A$ to abandon its strategy, independently from $R(t_{h+1})$. 
Letting $D(t_{h+1}) \colon \Omega \to \A$ denote the random variable (independent from $R(t_{h+1})$) representing the ``dead'' agent chosen to abandon its strategy at time $t_{h+1}$, we have
\[
\P\big(D(t_{h+1}) = n'\big) = \frac{1}{N} \quad \text{for every } n' \in \A \, .
\]
In particular, if $S(t_h) = \hat S$ is an observed realization of the random state of the system $S(t_h)$ at time $t_h$, then the probability that an agent with strategy $u_j$ is chosen to abandon its strategy is given by
\[ \label{eq:probability_abandonment_u_j}
\P\big(D(t_{h+1}) \in \hat \A_{u_j} \ \big|\  S(t_h) = \hat S \big) = \frac{\# \hat \A_{u_j}}{N} = \hat \lambda_{u_j} \, ,
\]
for $j = 1,\dots,M$.

\subsubsection*{Transition probabilities}  Let us fix $i, j \in \{1,\dots, M\}$ (possibly $i = j$).
Let us assume that at discrete time~$t_h$ we have observed a  state $S(t_h) = \hat S$. 
Let us compute the probability that the proportions transition from $\lambda(t_h) = \hat \lambda$ to $\lambda(t_{h+1}) = \hat \lambda + \frac{1}{N} (e_i -e_j)$, where $e_i, e_j \in \R^M$ are vectors of the standard basis of~$\R^M$. Writing more explicitly the updated proportions, we have that 
\[
  \begin{split}
    \hat \lambda + \frac{1}{N} ( e_i -  e_j) & = \Big(\hat \lambda_{u_1} , \dots, \underbrace{\hat \lambda_{u_i} + \frac{1}{N}}_{\text{place $i$}} , \dots, \underbrace{\hat \lambda_{u_j} - \frac{1}{N}}_{\text{place $j$}} , \dots, \hat \lambda_{u_M} \Big) \, ,
  \end{split}
\]
\ie, they are the proportions where the number of agents with strategy $u_i$ increases by 1 while the number of agents with strategy $u_j$ decreases by 1. 
This means that, given $\{\lambda(t_{h})  = \hat \lambda \}$, the event  $\{\lambda(t_{h+1})  = \hat \lambda + \frac{1}{N}( e_i -  e_j) \}$ occurs if and only if an agent $n$ with strategy $u_i$ is selected for replication and an agent $n'$ with strategy $u_j$ is chosen to abandon its strategy. 
In formulas, by independence, \eqref{eq:probability_replication_u_i} and~\eqref{eq:probability_abandonment_u_j}, we have that
\[ \label{eq:transition_probabilities}
\begin{split} 
  & \P\Big( \lambda(t_{h+1}) = \hat \lambda + \frac{1}{N}( e_i -  e_j) \ \Big|\  S(t_{h})  = \hat S \Big) = \P\big( \{ R(t_{h+1}) \in \hat \A_{u_i} \} \cap \{ D(t_{h+1}) \in \hat \A_{u_j} \} \ \big|\  S(t_{h})  = \hat S \big) \\
  & \quad = \P\big(  R(t_{h+1}) \in \hat \A_{u_i} \ \big|\  S(t_{h})  = \hat S \big) \P\big(  D(t_{h+1}) \in \hat \A_{u_j} \ \big|\  S(t_{h})  = \hat S \big) = \frac{1}{\overline f(\hat \lambda)} \hat \lambda_{u_i} f_{u_i}(\hat \lambda) \hat \lambda_{u_j}  \, . 
\end{split}
\]

The probability that the proportions remain unchanged is given by 
\[
\begin{split}
  & \P\big( \lambda(t_{h+1}) = \hat \lambda \ \big|\  S(t_{h})  = \hat S \big)  = 1 - \sum_{i=1}^M \sum_{\substack{j=1 \\ j \neq i}}^M \P\Big( \lambda(t_{h+1}) = \hat \lambda + \frac{1}{N}( e_i -  e_j) \ \Big|\  S(t_{h})  = \hat S \Big) \\
  & \quad = 1 - \sum_{i=1}^M \sum_{\substack{j=1 \\ j \neq i}}^M \frac{1}{\overline f(\hat \lambda)}\hat \lambda_{u_i} f_{u_i}(\hat \lambda) \hat \lambda_{u_j}  = 1 - \sum_{i=1}^M \sum_{j=1}^M \frac{\hat \lambda_{u_i} f_{u_i}(\hat \lambda) \hat \lambda_{u_j} }{\sum_{\ell=1}^M   \hat \lambda_{u_\ell}  f_{u_\ell}(\hat \lambda)} + \frac{\sum_{i=1}^M \hat \lambda_{u_i}^2 f_{u_i}(\hat \lambda) }{\sum_{\ell=1}^M   \hat \lambda_{u_\ell}  f_{u_\ell}(\hat \lambda)} \\ 
  & \quad = 1 - \sum_{i=1}^M \hat \lambda_{u_i} + \frac{\sum_{i=1}^M \hat \lambda_{u_i}^2 f_{u_i}(\hat \lambda) }{\sum_{\ell=1}^M   \hat \lambda_{u_\ell}  f_{u_\ell}(\hat \lambda)} = \frac{1}{\overline f(\hat \lambda)} \sum_{i=1}^M \hat \lambda_{u_i}^2 f_{u_i}(\hat \lambda)  \, .
\end{split}
\]

We concluded the description of the discrete stochastic process. 
We are interested in the large-population limit of the process as $N \to \infty$ and the time step $\tau_k \to 0$.
Before proceeding with the analysis, we introduce the main tool that we will use to study the limit of the process.

 
\section{Eulerian specification of the discrete stochastic process} \label{sec:Eulerian_description}
 
In this section we introduce a probability measure $\Lambda^k_t$ that we will use to describe the discrete stochastic process with an Eulerian point of view. 
After deriving some properties of this measure, we will show that $\Lambda^k_t$ is an approximate solution of a suitable continuity equation. 
This will also allow us to study the limit of the process as $N \to \infty$ and $\tau_k \to 0$.

\subsubsection*{Sequences $N_k$ and $w_k$} From now on, we will consider 
\[
N_k \sim \tau_k^{-\alpha} \, , \quad \text{for some } \alpha > 0 \, , \quad \quad w_k \sim \tau_k^\beta \, , \quad \text{for some } \beta > 0 \, ,
\]
where we recall that $\tau_k = \frac{T}{k}$ is the time step of the process.
Since we are interested in the limit behavior of the process as $k \to +\infty$, we will stress the dependence on $k$ of $\lambda^k(t_h)$, \ie, the proportions at time $t_h$ in the process described in Section~\ref{sec:discrete_stochastic_process} with population size $N_k$ and time step~$\tau_k$.

\subsubsection*{Piecewise affine interpolation} First of all we extend the discrete stochastic process to all times $t \in [0,T]$ by defining $\lambda^k(t) \colon \Omega \to \Delta^{M-1}$ by piecewise affine interpolation, \ie, we set 
\[ \label{eq:piecewise_affine_interpolation}
\lambda^k(t) =   \lambda^k(t_h) +  \frac{t - t_h}{\tau_k} ( \lambda^k(t_{h+1}) - \lambda^k(t_h) ) \in \Delta^{M-1} \, , \quad \text{for } t \in [t_h, t_{h+1}] \, ,
\]
for $h = 0, \dots, k-1$. 
For every realization $\omega \in \Omega$ of the probability space, we obtain a continuous path $\lambda^k(\cdot)(\omega) \in C([0,T],\Delta^{M-1})$. 
Hence, we have defined a random continuous path 
\[
\lambda^k(\cdot) \colon \Omega \to C([0,T],\Delta^{M-1}) \, .
\] 

\subsubsection*{Distribution of proportions} We define the probability measure on continuous paths
\[ \label{eq:Lambda_k}
\Lambda^k := \lambda^k(\cdot)_\# \P \in \Prob\big(C([0,T],\Delta^{M-1})\big) \, ,
\]
where $\lambda^k(\cdot)_\# \P$ denotes the pushforward of $\P$ through the random continuous path $\lambda^k(\cdot)$. 
The probability measure $\Lambda^k$ describes how the evolution of the random proportions $\lambda^k(t)$ is distributed over the space of continuous paths.

For every $t \in [0,T]$, we consider the evaluation map on continuous paths, \ie, 
\[
\ev_t \colon C([0,T],\Delta^{M-1}) \to \Delta^{M-1} \, , \quad \lambda \mapsto \ev_t(\lambda) = \lambda(t) \, .
\]
With this map, we define the probability measure 
\[ \label{eq:Lambda_k_t}
\Lambda^k_t = (\ev_t)_\# \Lambda^k = (\ev_t)_\# \lambda^k(\cdot)_\# \P = (\ev_t \circ \lambda^k(\cdot))_\# \P = \lambda^k(t)_\# \P  \in \Prob(\Delta^{M-1}) \, .
\]
This probability measure can be regarded as the distribution of the random proportions $\lambda^k(t)$ at a given time $t$.

\begin{remark}
  If one adopts the identification of $\Delta^{M-1}$ with the set of probability measures on $\U$, one can interpret $\Lambda^k_t$ as an element of $\Prob(\Prob(\U))$.
\end{remark}

\subsubsection*{Piecewise constant interpolation} We introduce the notation $\overline \Lambda^k_t$ for the piecewise constant curve defined by
\[ \label{eq:Lambda_bar_k_t}
\overline \Lambda^k_t = \Lambda^k_{t_h} \quad \text{for } t \in [t_h,t_{h+1}) \, , \quad \overline \Lambda^k_T = \Lambda^k_T \, .
\] 
In the next lemma we show that $\overline \Lambda^k_t$ and $\Lambda^k_t$ are close in the 1-Wasserstein distance as $k \to +\infty$.

\begin{lemma} \label{lem:Lambda_k_Lambda_bar}
  Let $\Lambda^k_t$ be as in~\eqref{eq:Lambda_k_t} and $\overline \Lambda^k_t$ be as in~\eqref{eq:Lambda_bar_k_t}.
   We have that $\W_1(\Lambda^k_t,\overline \Lambda^k_t) \to 0$ as $k \to +\infty$ for every $t \in [0,T]$.
\end{lemma}
\begin{proof}
  If $t = T$ there is nothing to prove. 
  Let us fix $t \in [t_h,t_{h+1})$ for some $h \in \{0, \dots, k-1\}$ and let us fix $\psi \colon \Delta^{M-1} \to \R$ a 1-Lipschitz function.
  By the definition of $\Lambda^k$ and $\overline \Lambda^k$, we have that
  \[ \label{eq:psi_Lambda_Lambda_bar}
  \begin{split}
    & \int_{\Delta^{M-1}} \psi(\lambda) \, \d \Lambda^k_t(\lambda) - \int_{\Delta^{M-1}} \psi(\lambda) \, \d \overline \Lambda^k_t(\lambda) \\
    & \quad = \int_{\Delta^{M-1}} \psi(\lambda) \, \d \big( \Lambda^k_t -  \Lambda^k_{t_h} \big)(\lambda)  = \int_{\Delta^{M-1}} \psi(\lambda) \, \d \big( \lambda^k(t) - \lambda^k(t_h) \big)_\# \P(\lambda) \\
    & \quad = \E\big[ \psi(\lambda^k(t)) - \psi(\lambda^k(t_h)) \big] \leq \E\big[ |\lambda^k(t) - \lambda^k(t_h)| \big] \, .
  \end{split}
  \]
  By~\eqref{eq:piecewise_affine_interpolation}, we have that
  \[ \label{eq:estimate_piecewise_affine_with_increment}
  \E\big[ |\lambda^k(t) - \lambda^k(t_h)| \big] = \frac{t-t_h}{\tau_k} \E\big[ |\lambda^k(t_{h+1}) - \lambda^k(t_h)| \big] \leq \E\big[ |\lambda^k(t_{h+1}) - \lambda^k(t_h)| \big] \, .
  \]
  To estimate the right-hand side, we apply the law of total probability and take the sum over all possible states of the system at time $t_h$:
  \[ \label{eq:total_probability}
    \E\big[|\lambda^k(t_{h+1}) - \lambda^k(t_h)|\big] = \sum_{\hat S} \E\big[|\lambda^k(t_{h+1}) - \hat \lambda| \ \big|\ S(t_h) = \hat S \big] \P\big(S(t_h) = \hat S\big) \, .
  \]
  Given $\lambda^k(t_h) = \hat \lambda$, the proportions $\lambda^k(t_{h+1})$ either remain unchanged or one component of the proportions increases by $\frac{1}{N_k}$ and another decreases by $\frac{1}{N_k}$.
  Hence,
  \[
    \E\big[|\lambda^k(t_{h+1}) - \hat \lambda| \ \big|\ S(t_h) = \hat S \big] \leq \frac{\sqrt{2}}{N_k} \, .
  \]
  By~\eqref{eq:psi_Lambda_Lambda_bar}--\eqref{eq:total_probability}, we infer that
  \[
  \int_{\Delta^{M-1}} \psi(\lambda) \, \d \Lambda^k_t(\lambda) - \int_{\Delta^{M-1}} \psi(\lambda) \, \d \overline \Lambda^k_t(\lambda) \leq \frac{\sqrt{2}}{N_k} \, , \quad \text{for every 1-Lipschitz function } \psi \, .
  \]
  Taking the supremum over all 1-Lipschitz functions $\psi$, by Kantorovich's duality we obtain that $\W_1(\Lambda^k_t,\overline \Lambda^k_t) \leq \frac{\sqrt{2}}{N_k} \to 0$ as $k \to +\infty$, concluding the proof.
\end{proof}

\subsubsection*{Continuity equation} To describe the Eulerian evolution of the process, it is convenient to introduce the vector field $b = (b_1,\dots,b^M)^\top \colon \Delta^{M-1} \to \R^M$ defined by
\[ \label{eq:b}
b_i(\hat \lambda) := \hat \lambda_{u_i} \big( (A \hat \lambda)_i - (A \hat \lambda)\cdot\hat \lambda \big) \, , \quad \text{for } i = 1, \dots, M \, .
\]
\begin{remark}
  The vector field $b$ is tangent to the simplex $\Delta^{M-1}$, \ie, $b(\hat \lambda) \in \mathrm{Tan}(\Delta^{M-1},\hat \lambda)$ for every $\hat \lambda \in \Delta^{M-1}$.
  Indeed, recalling that $e_1+\dots+e_M$ is the normal to $\Delta^{M-1}$ at every $\hat \lambda$, we have that
  \[
  \sum_{i=1}^M b_i(\hat \lambda) = \sum_{i=1}^M \hat \lambda_{u_i} \big( (A \hat \lambda)_i - (A \hat \lambda)\cdot \hat \lambda \big) = (A \hat \lambda)\cdot \hat \lambda - (A \hat \lambda)\cdot \hat \lambda = 0 \, .
  \]
\end{remark}
 
We have the following result, in which we show that, in a suitable sense, $\Lambda^k$ solves is an approximate distributional solution to the continuity equation
\[
\de_t \Lambda^k_t + \div \big(  b  \Lambda^k_t \big) \approx 0 \, .
\]

\begin{proposition} \label{prop:approximate_distributional_solution}
  Let $\Lambda^k_t$ be as in~\eqref{eq:Lambda_k_t} and $\overline \Lambda^k_t$ be as in~\eqref{eq:Lambda_bar_k_t}. 
  Let $\phi \in C^\infty_c((-\infty,T) \times \R^M)$. 
  Then 
  \[ \label{eq:Lambda_k_distributional_solution}
  \int_0^T \int_{\Delta^{M-1}} \de_t \phi(t,\lambda)  \, \d \Lambda^k_t(\lambda) \, \d t + \int_0^T \int_{\Delta^{M-1}} \D \phi(t,\lambda) b(\lambda)   \, \d \overline \Lambda^k_t(\lambda) \, \d t  = - \int_{\Delta^{M-1}} \phi(0,\lambda) \, \d \Lambda^k_0(\lambda) + \tilde \rho_k \, ,
  \]
  where 
  \[
  \tilde \rho_k = \O\Big(\frac{w_k}{\tau_k N_k^2}   \|\D\phi\|_\infty \Big) + \O\big(\frac{w_k^2}{\tau_k N_k}  \|\D \phi\|_\infty  \big)  + \O\Big( \frac{1}{\tau_k N_k^2}  \|\D^2 \phi\|_{L^\infty} \Big) \, .
  \]
\end{proposition}

Before proving the proposition, we need a technical estimate.

\begin{lemma} \label{lem:E_psi}
  Let $\Lambda^k_t$ be as in~\eqref{eq:Lambda_k_t} and $\overline \Lambda^k_t$ be as in~\eqref{eq:Lambda_bar_k_t}. 
  Let $\psi \in C^\infty_c(\R^M)$. 
  Then 
  \[ \label{eq:E_psi}
    \E\big[ \D\psi(\lambda^k(t_h))[\lambda^k(t_{h+1}) - \lambda^k(t_h)]\big] = \frac{w_k}{N_k}   \int_{\Delta^{M-1}} \D \psi(\lambda) b(\lambda) \, \d \Lambda^k_{t_h}(\lambda) + \rho_k \, ,
  \]
  for $h = 0, \dots, k-1$, where 
  \[
  \rho_k = \O\Big( \frac{w_k}{N_k^2}  \|\D\psi\|_\infty \Big) + \O\Big( \frac{w_k^2}{N_k}  \|\D \psi\|_\infty  \Big)  \, .
  \]
\end{lemma}
\begin{proof}
  By the law of total probability:
  \[ \label{eq:expectation_first_term_law_total_probability}
  \E\big[\D \psi(\lambda^k(t_h)) [ \lambda^k(t_{h+1}) - \lambda^k(t_h)] \big] = \sum_{\hat S} \E\big[ \D \psi(\hat \lambda) [ \lambda^k(t_{h+1}) - \hat \lambda ] \ \big|\ S(t_h) = \hat S \big] \P\big(S(t_h) = \hat S\big) \, . 
  \]
  Then, by~\eqref{eq:transition_probabilities}, we obtain that
  \[ \label{eq:expectation_first_term}
  \begin{split}
    & \E\big[ \D \psi(\hat \lambda) \big[ \lambda^k(t_{h+1}) - \hat \lambda \big] \ \big|\ S(t_h) = \hat S \big] \\
    & \quad = \sum_{i=1}^M \sum_{j=1}^M \D \psi(\hat \lambda) \Big[ \frac{1}{N_k}(e_i - e_j) \Big] \P\Big( \lambda^k(t_{h+1}) = \hat \lambda + \frac{1}{N_k}(e_i - e_j) \ \big|\ S(t_h) = \hat S \Big) \\
    & \quad = \frac{1}{N_k}\sum_{i=1}^M \sum_{j=1}^M \big( \de_i \psi(\hat \lambda) - \de_j \psi(\hat \lambda) \big) \frac{1}{\overline f(\hat \lambda)} \hat \lambda_{u_i} f_{u_i}(\hat \lambda) \hat \lambda_{u_j} \, . 
  \end{split}
  \]
  Using the fact that $\sum_{j=1}^M \hat \lambda_{u_j} = 1$, we infer that
  \[
    \frac{1}{N_k}\sum_{i=1}^M \sum_{j=1}^M  \de_i \psi(\hat \lambda)  \frac{1}{\overline f(\hat \lambda)} \hat \lambda_{u_i} f_{u_i}(\hat \lambda) \hat \lambda_{u_j} = \frac{1}{N_k} \sum_{i=1}^M \de_i \psi(\hat \lambda) \frac{1}{\overline f(\hat \lambda)} \hat \lambda_{u_i} f_{u_i}(\hat \lambda)  \, , 
  \]
  while, by the definition of $\overline f(\hat \lambda)$ in~\eqref{eq:average_fitness}, we have that
  \[
  \begin{split}
    \frac{1}{N_k}\sum_{i=1}^M \sum_{j=1}^M  \de_j \psi(\hat \lambda)  \frac{1}{\overline f(\hat \lambda)} \hat \lambda_{u_i} f_{u_i}(\hat \lambda) \hat \lambda_{u_j} & = \frac{1}{N_k}\sum_{j=1}^M \de_j \psi(\hat \lambda) \hat \lambda_{u_j}  \sum_{i=1}^M  \frac{1}{\overline f(\hat \lambda)} \hat \lambda_{u_i} f_{u_i}(\hat \lambda)  = \frac{1}{N_k} \sum_{j=1}^M \de_j \psi(\hat \lambda)  \hat \lambda_{u_j} \\
    & = \frac{1}{N_k} \sum_{i=1}^M \de_i \psi(\hat \lambda)  \hat \lambda_{u_i}   \, .
  \end{split}
    \]
  Inserting these two expressions in~\eqref{eq:expectation_first_term}, we deduce that
  \[ \label{eq:expectation_first_term_final}
  \begin{split}
    \E\big[\D \psi(\hat \lambda) [ \lambda^k(t_{h+1}) - \hat \lambda ] \ \big|\ S(t_h) = \hat S \big] & = \frac{1}{N_k} \sum_{i=1}^M \de_i \psi(\hat \lambda) \frac{1}{\overline f(\hat \lambda)}  \hat \lambda_{u_i} f_{u_i}(\hat \lambda) - \frac{1}{N_k} \sum_{i=1}^M \de_i \psi(\hat \lambda) \hat \lambda_{u_i} \\
    & = \frac{1}{N_k} \sum_{i=1}^M \de_i \psi(\hat \lambda) \frac{1}{\overline f(\hat \lambda)}  \hat \lambda_{u_i} (f_{u_i}(\hat \lambda) - \overline f(\hat \lambda)) \, .
  \end{split}
  \]
  On the one hand, by~\eqref{eq:average_fitness}, \eqref{eq:fitness}, and~\eqref{eq:average_payoff} we have that 
  \[ \label{eq:average_fitness_explicit}
  \begin{split}
    \overline f(\hat \lambda) & = \sum_{\ell=1}^M \hat \lambda_{u_\ell} f_{u_\ell}(\hat \lambda) = \sum_{\ell=1}^M \hat \lambda_{u_\ell} (1-w_k + w_k \pi_{u_\ell}(\hat \lambda)) = 1 - w_k + w_k \sum_{\ell=1}^M \hat \lambda_{u_\ell} \pi_{u_\ell}(\hat \lambda) \\
    & = 1 - w_k + w_k \sum_{\ell=1}^M \hat \lambda_{u_\ell} \Big( \frac{N_k}{N_k-1} (A \hat \lambda)_\ell - \frac{1}{N_k - 1} a_{\ell \ell} \Big) \\
    & = 1 - w_k + w_k \frac{N_k}{N_k - 1} (A \hat \lambda) \cdot \hat \lambda -  \frac{w_k}{N_k - 1} \diag(A) \cdot \hat \lambda \, ,
  \end{split}
  \]
  where $\diag(A) \in \R^M$ denotes the vector with the diagonal elements of $A$. 
  On the other hand, by~\eqref{eq:fitness}
  \[
  f_{u_i}(\hat \lambda) = 1 - w_k + w_k \pi_{u_i}(\hat \lambda) = 1 - w_k + w_k   \frac{N_k}{N_k-1} (A \hat \lambda)_i - \frac{w_k}{N_k - 1} a_{ii}  \, .
  \]
  We substitute these expressions in~\eqref{eq:expectation_first_term_final} to obtain that
  \[ \label{eq:expectation_first_term_final_final}
  \begin{split}
    \E\big[\D \psi(\hat \lambda) [ \lambda^k(t_{h+1}) - \hat \lambda ] \ \big|\ S(t_h) = \hat S \big] & = \frac{w_k}{N_k} \frac{N_k}{N_k-1} \frac{1}{\overline f(\hat \lambda)} \sum_{i=1}^M \de_i \psi(\hat \lambda) \hat \lambda_{u_i} \Big( (A \hat \lambda)_i - (A \hat \lambda) \cdot \hat \lambda \Big) \\
    & \quad + \frac{w_k}{N_k} \frac{1}{N_k - 1} \frac{1}{\overline f(\hat \lambda)} \sum_{i=1}^M \de_i \psi(\hat \lambda) \hat \lambda_{u_i} \Big( \diag(A) \cdot \hat \lambda - a_{ii} \Big) \\
    &  = \frac{w_k}{N_k} \frac{1}{\overline f(\hat \lambda)} \sum_{i=1}^M \de_i \psi(\hat \lambda) b_i(\hat \lambda)  + \frac{w_k}{N_k} R_{1,k}(\hat \lambda) \\
    &  = \frac{w_k}{N_k} \frac{1}{\overline f(\hat \lambda)} \D \psi(\hat \lambda) b(\hat \lambda)  + \frac{w_k}{N_k} R_{1,k}(\hat \lambda) \, ,
  \end{split}
  \]
  where 
  \[ \label{eq:remainder_first_term}
  R_{1,k}(\hat \lambda) = \frac{1}{N_k - 1} \frac{1}{\overline f(\hat \lambda)} \sum_{i=1}^M \de_i \psi(\hat \lambda) \hat \lambda_{u_i} \Big( (A \hat \lambda)_i - (A \hat \lambda) \cdot \hat \lambda+\diag(A) \cdot \hat \lambda - a_{ii} \Big) \, .
  \]
  Let us estimate $R_{1,k}$. 
  To bound $\frac{1}{\overline f(\hat \lambda)}$, we set
  \[
  g_k(\hat \lambda) := -1 + \frac{N_k}{N_k - 1} (A \hat \lambda) \cdot \hat \lambda -  \frac{1}{N_k - 1} \diag(A) \cdot \hat \lambda \, .
  \]
  We start by observing that
  \[ \label{eq:average_fitness_g}
  |g_k(\hat \lambda)| \leq c_0 \, , \quad \text{for every } \hat \lambda \in \Delta^{M-1} \, , \ k \geq 1 \, ,
  \]
  for some constant $c_0 > 0$ depending on $A$.
  This implies that
  \[ \label{eq:average_fitness_inverse_bound}
  \frac{1}{\overline f(\hat \lambda)} \leq c_1 \, , \quad \text{for every } \hat \lambda \in \Delta^{M-1} \, , \ k \text{ large enough} \, ,
  \]
  for some constant $c_1 > 0$, where $k \geq 1$ large enough means, \eg, $w_k \leq \frac{1}{2c_0}$.

  We use~\eqref{eq:average_fitness_inverse_bound} to estimate the remainder $R_{1,k}$ in~\eqref{eq:remainder_first_term} by 
  \[ \label{eq:remainder_1_first_term_bound}
  |R_{1,k}(\hat \lambda)| \leq C_1 \|\D\psi\|_{L^\infty} \frac{1}{N_k} \, , \quad \text{for every } \hat \lambda \in \Delta^{M-1} \, , \ k \text{ large enough} \, ,
  \]
  where the constant $C_1 > 0$ depends on $A$ and $c_1$.

  We recast the term $\frac{1}{\overline f(\hat \lambda)}$ in~\eqref{eq:expectation_first_term_final_final} in a more convenient form.
  Exploiting once more~\eqref{eq:average_fitness_g}, from
  \[
    \frac{1}{\overline f(\hat \lambda)} = \frac{1}{1 + w_k g_k(\hat \lambda)} = 1 -  \frac{w_k g_k(\hat \lambda)}{1 + w_k g_k(\hat \lambda)} \, ,
  \]
  we also deduce that 
  \[ \label{eq:average_fitness_inverse_approximation}
  \Big| \frac{1}{\overline f(\hat \lambda)} - 1 \Big| \leq  \frac{w_k g_k(\hat \lambda)}{|1 + w_k g_k(\hat \lambda)|} \leq c_2 w_k \, , \quad \text{for every } \hat \lambda \in \Delta^{M-1} \, ,  \ k \text{ large enough} \, ,
  \]
  for some constant $c_2 > 0$, where $k \geq 1$ large enough means, \eg, $w_k \leq \frac{1}{2c_0}$.

  Then we recast~\eqref{eq:expectation_first_term_final_final} as
  \[ \label{eq:expectation_first_term_final_final_final}
    \E\big[\D \psi(\hat \lambda) [ \lambda^k(t_{h+1}) - \hat \lambda ] \ \big|\ S(t_h) = \hat S \big] = \frac{w_k}{N_k}\Big(  \D \psi(\hat \lambda) b(\hat \lambda)  +   R_{1,k}(\hat \lambda) +   R_{2,k}(\hat \lambda) \Big) \, ,
  \]
  where
  \[ \label{eq:remainder_first_term_final}
  R_{2,k}(\hat \lambda) =  \Big( \frac{1}{\overline f(\hat \lambda)} - 1 \Big) \sum_{i=1}^M \de_i \psi(\hat \lambda) b_i(\hat \lambda)  \, .
  \]
  By~\eqref{eq:average_fitness_inverse_approximation}, we estimate the remainder $R_{2,k}$ by 
  \[   \label{eq:remainder_2_first_term_bound}
  |R_{2,k}(\hat \lambda)| \leq C_2 \|\D\psi\|_{L^\infty} w_k \, , \quad \text{for every } \hat \lambda \in \Delta^{M-1} \, , \ k \text{ large enough} \, ,
  \]
  where the constant $C_2 > 0$ depends on $A$ and $c_2$.

  We have concluded the estimate.
  Indeed, by~\eqref{eq:expectation_first_term_law_total_probability}, \eqref{eq:expectation_first_term_final_final_final}, \eqref{eq:remainder_1_first_term_bound}, and~\eqref{eq:remainder_2_first_term_bound}, we deduce that
  \[  
  \begin{split}
    & \E\big[\D \psi(\lambda^k(t_h)) [ \lambda^k(t_{h+1}) - \lambda^k(t_h)] \big] = \frac{w_k}{N_k} \sum_{\hat S} \Big(  \D \psi(\hat \lambda) b(\hat \lambda)  +   R_{1,k}(\hat \lambda) +   R_{2,k}(\hat \lambda) \Big) \P\big(S(t_h) = \hat S\big)  \\
    & \quad  = \frac{w_k}{N_k} \Big(  \E\big[ \D \psi(\lambda^k(t_{h})) b(\lambda^k(t_h)) \big] + \E\big[ R_{1,k}(\lambda^k(t_h)) \big] + \E\big[ R_{2,k}(\lambda^k(t_h)) \big] \Big) \\ 
    & \quad = \frac{w_k}{N_k} \Big(  \E\big[ \D \psi(\lambda^k(t_{h})) b(\lambda^k(t_h)) \big] + \O\Big(\frac{\|\D\psi\|_{L^\infty}}{N_k}\Big) + \O\big(\|\D\psi\|_{L^\infty}w_k\big) \Big) \, .
  \end{split}
  \]
  We note that
  \[
  \begin{split}
    \E\big[ \D \psi(\lambda^k(t_{h})) b(\lambda^k(t_h)) \big] & = \int_{\Omega} \D \psi(\lambda^k(t_h)(\omega)) b(\lambda^k(t_h)(\omega)) \, \d \P(\omega) = \int_{\Delta^{M-1}} \D \psi(\lambda) b(\lambda) \, \d \lambda^k(t_h)_\# \P(\lambda) \\
    & = \int_{\Delta^{M-1}} \D \psi(\lambda) b(\lambda) \, \d \Lambda^k_{t_h}(\lambda) \, ,
  \end{split}
  \]
  hence 
  \[ \label{eq:first_term}
  \begin{split}
    \E\big[\D \psi(\lambda^k(t_h)) \big[ \lambda^k(t_{h+1}) - \lambda^k(t_h)\big] \big] & = \frac{w_k}{N_k}\Big( \int_{\Delta^{M-1}}   \D \psi(\lambda) b(\lambda) \, \d \Lambda^k_{t_h}(\lambda) \\ 
    & \hphantom{\frac{w_k}{N_k}\Big(\int_{\Delta^{M-1}}}+ \O\Big(\frac{\|\D\psi\|_{L^\infty}}{N_k}\Big) + \O\big(\|\D\psi\|_{L^\infty}w_k\big) \Big) \, .
  \end{split} 
  \]
  This concludes the proof of the lemma.
\end{proof}

We are now in a position to prove Proposition~\ref{prop:approximate_distributional_solution}.

\begin{proof}[Proof of Proposition~\ref{prop:approximate_distributional_solution}]
  Let us fix $\phi \in C^\infty_c((-\infty,T) \times \R^M)$. 
  By~\eqref{eq:piecewise_affine_interpolation}, for a.e.\ realization $\omega \in \Omega$, we have that
  \[ \label{eq:derivative_phi}
  \begin{split}
    \frac{\d}{\d t} \big( \phi(t,\lambda^k(t)(\omega)) \big) & = \de_t \phi(t,\lambda^k(t)(\omega)) + \D \phi(t,\lambda^k(t)(\omega)) \Big[ \frac{\d}{\d t} \lambda^k(t)(\omega) \Big] \\
    & = \de_t \phi(t,\lambda^k(t)(\omega)) + \D \phi(t,\lambda^k(t)(\omega)) \Big[ \frac{\lambda^k(t_{h+1})(\omega) - \lambda^k(t_h)(\omega)}{\tau_k} \Big]
  \end{split}
  \]
  for every $t \in (t_h,t_{h+1})$.
  We take the expectation and we integrate in $t$ the above equality.
  The left-hand side in~\eqref{eq:derivative_phi} reads
  \[  \label{eq:derivative_phi_left}
  \begin{split}
    \int_{t_h}^{t_{h+1}} \frac{\d}{\d t} \int_{\Omega} \phi(t,\lambda^k(t)(\omega)) \, \d \P(\omega) \, \d t & = \int_{\Omega} \phi(t_{h+1},\lambda^k(t_{h+1})(\omega)) \, \d \P(\omega) - \int_{\Omega} \phi(0,\lambda^k(0)(\omega)) \, \d \P(\omega) \\
    & = \int_{\Delta^{M-1}} \phi(t_{h+1},\lambda) \, \d \Lambda^k_{t_{h+1}}(\lambda) - \int_{\Delta^{M-1}} \phi(t_h,\lambda) \, \d \Lambda^k_{t_h}(\lambda) \, .
  \end{split}
  \]
  The right-hand side in~\eqref{eq:derivative_phi} reads
  \[ \label{eq:derivative_phi_right}
  \begin{split}
    & \int_{t_h}^{t_{h+1}} \int_{\Omega} \de_t \phi(t,\lambda^k(t)(\omega)) \, \d \P(\omega) \, \d t + \int_{t_h}^{t_{h+1}} \int_{\Omega} \D \phi(t,\lambda^k(t)(\omega)) \Big[ \frac{\lambda^k(t_{h+1})(\omega) - \lambda^k(t_h)(\omega)}{\tau_k} \Big] \, \d \P(\omega) \, \d t \\
    & = \int_{t_h}^{t_{h+1}} \int_{\Delta^{M-1}} \de_t \phi(t,\lambda) \, \d \Lambda^k_t(\lambda) \, \d t + \frac{1}{\tau_k}\int_{t_h}^{t_{h+1}}  \E \big[ \D \phi(t,\lambda^k(t)) \big[  \lambda^k(t_{h+1}) - \lambda^k(t_h) \big] \big]   \, \d t \, .
  \end{split}
  \]
  We now estimate the term $\E \big[ \D \phi(t,\lambda^k(t)) \big[  \lambda^k(t_{h+1}) - \lambda^k(t_h) \big] \big]$.
  To do so, we use \eqref{eq:piecewise_affine_interpolation} and Taylor's formula to write a.s.\ in $\Omega$
  \[ \label{eq:Taylor_derivative}
  \begin{split}
    & \D \phi(t,\lambda^k(t))\big[  \lambda^k(t_{h+1}) - \lambda^k(t_h) \big] = \D \phi\Big(t,\lambda^k(t_h) +  (\lambda^k(t) - \lambda^k(t_h))) \Big)\big[  \lambda^k(t_{h+1}) - \lambda^k(t_h) \big]  \\
    & \quad =  \D \phi(t,\lambda^k(t_h))\big[  \lambda^k(t_{h+1}) - \lambda^k(t_h) \big] + R(t,\lambda^k(t_h),\lambda^k(t_{h+1})) \, ,
  \end{split}
  \]
  where 
  \[
  \begin{split}
    & R(t,\lambda^k(t_h),\lambda^k(t_{h+1})) \\
    &  = \frac{t-t_h}{\tau_k}   \int_0^1  \D^2 \phi\Big(t,\lambda^k(t_h) + r  \big(\lambda^k(t) - \lambda^k(t_h) \big) \Big)  \big[\lambda^k(t_{h+1}) - \lambda^k(t_h)), \lambda^k(t_{h+1}            ) - \lambda^k(t_h)  \big] \, \d r   . 
  \end{split}
  \]
  Since $t$ is fixed, we apply Lemma~\ref{lem:E_psi} to $\psi = \phi(t,\cdot)$ to get that 
  \[\label{eq:expectation_derivative_estimate}
  \begin{split}
    & \E\big[ \D\phi(t,\lambda^k(t_h))[\lambda^k(t_{h+1}) - \lambda^k(t_h)]\big] \\
    & \quad = \frac{w_k}{N_k}   \int_{\Delta^{M-1}} \D \phi(t,\lambda) b(\lambda) \, \d \Lambda^k_{t_h}(\lambda) + \O\Big( \frac{w_k}{N_k^2}  \|\D\psi\|_\infty \Big) + \O\Big( \frac{w_k^2}{N_k}  \|\D \psi\|_\infty  \Big) \, .
  \end{split}
  \]
  Moreover, we use the law of total probability to estimate 
  \[
    \E\big[R(t,\lambda^k(t_h),\lambda^k(t_{h+1})) \big] = \sum_{\hat S} \E\big[ R(t,\hat \lambda,\lambda^k(t_{h+1})) \ \big|\ S(t_h) = \hat S \big] \P\big(S(t_h) = \hat S\big) \, ,
  \]
  and 
  \[
  \E\big[R(t,\hat \lambda,\lambda^k(t_{h+1}))  \ \big|\ S(t_h) = \hat S \big] \leq \| \D^2 \phi\|_{L^\infty} \frac{2}{N_k^2} = \O\Big(\frac{\|\D^2 \phi\|_{L^\infty}}{N_k^2}\Big) \, ,
  \]
  from which we deduce that
  \[ \label{eq:expectation_remainder_estimate}
  \E\big[R(t,\lambda^k(t_h),\lambda^k(t_{h+1})) \big] = \O\Big(\frac{\|\D^2 \phi\|_{L^\infty}}{N_k^2}\Big) \, .
  \]

  Combining~\eqref{eq:derivative_phi}--\eqref{eq:expectation_remainder_estimate}, we infer that
  \[ \label{eq:derivative_phi_final}
  \begin{split}
    & \int_{\Delta^{M-1}} \phi(t_{h+1},\lambda) \, \d \Lambda^k_{t_{h+1}}(\lambda) - \int_{\Delta^{M-1}} \phi(t_h,\lambda) \, \d \Lambda^k_{t_h}(\lambda) = \\
    & \quad =  \int_{t_h}^{t_{h+1}} \int_{\Delta^{M-1}} \de_t \phi(t,\lambda) \, \d \Lambda^k_t(\lambda) \, \d t + \frac{w_k}{\tau_k N_k}  \int_{t_h}^{t_{h+1}} \int_{\Delta^{M-1}} \D \phi(t,\lambda) b(\lambda) \, \d \Lambda^k_{t_h}(\lambda) \\
    & \quad \quad + \O\Big( \frac{w_k}{N_k^2}  \|\D\psi\|_\infty \Big) + \O\Big( \frac{w_k^2}{N_k}  \|\D \psi\|_\infty  \Big)  + \O\Big(\frac{1}{N_k^2} \|\D^2 \phi\|_{L^\infty}\Big) \, .
  \end{split}
  \]
  Summing the above equality over $h = 0, \dots, k-1$, we obtain that
  \[ \label{eq:derivative_phi_final_final}
  \begin{split}
    - \int_{\Delta^{M-1}} \phi(0,\lambda) \, \d \Lambda^k_{0}(\lambda) = & \int_{\Delta^{M-1}} \phi(T,\lambda) \, \d \Lambda^k_{T}(\lambda) - \int_{\Delta^{M-1}} \phi(0,\lambda) \, \d \Lambda^k_{0}(\lambda) = \\
    & =  \int_{0}^{T} \int_{\Delta^{M-1}} \de_t \phi(t,\lambda) \, \d \Lambda^k_t(\lambda) \, \d t + \frac{w_k}{\tau_k N_k}   \int_{\Delta^{M-1}} \D \phi(t,\lambda) b(\lambda) \, \d \overline \Lambda^k_{t}(\lambda) \\
    & \quad + \O\Big(\frac{w_k}{\tau_k N_k^2}   \|\D\psi\|_\infty \Big) + \O\big(\frac{w_k^2}{\tau_k N_k}  \|\D \psi\|_\infty  \big)  + \O\Big( \frac{1}{\tau_k N_k^2}  \|\D^2 \phi\|_{L^\infty} \Big) \, .
  \end{split}
  \]
  This concludes the proof of the proposition.

\end{proof}

\section{Continuous-time large population limit}

We are in a position to prove the main result in this paper, in which we characterize the limit of the measures $\Lambda^k$ as $k \to +\infty$.

First of all, we need to establish a compactness result for the sequence $\Lambda^k$.
For this, we start with a preliminary observation.

\begin{remark} \label{rem:compactness_preliminaries}
  Let $\Lambda^k_t$ be as in~\eqref{eq:Lambda_k_t}.
  We have that:
  \begin{enumerate}[label={\itshape (\roman*)}]
    \item \label{item:values_in_P1} $\Lambda^k_t \in \Prob_1(\Delta^{M-1})$ for every $t \in [0,T]$;
    \item \label{item:P1_compact} $(\Prob_1(\Delta^{M-1}), \W_1)$ is a compact metric space.
  \end{enumerate}
For~\ref{item:values_in_P1}, we fix $e_1 \in \Delta^{M-1}$ and we obtain, by boundedness of $\Delta^{M-1}$, that
  \[
  \int_{\Delta^{M-1}} |e_1-\lambda| \, \d \Lambda^k_t(\lambda) \leq C \int_{\Delta^{M-1}} \d \Lambda^k_t(\lambda) \leq C \, ,
  \]
  for some constant $C > 0$. 
  In fact, the estimate is uniform in $t \in [0,T]$, $N$, and $\tau_k$.

  Fact~\ref{item:P1_compact} is a known result, since $\Delta^{M-1}$ is compact. 
  See, \eg, \cite[Remark~6.19]{Vil08}.
\end{remark}

\begin{remark} \label{rem:nontrivial_compactness}
  If the curves $t \in [0,T] \mapsto \Lambda^k_t \in \Prob_1(\Delta^{M-1})$ were equicontinuous (for example if they were equi-Lipschitz continuous), then one could apply Arzelà--Ascoli's theorem to extract a subsequence converging to a continuous curve $t \in [0,T] \mapsto \Lambda_t \in \Prob_1(\Delta^{M-1})$.
  However, straightforward estimates coming from the discrete stochastic process do not seem to provide equicontinuity.
  Indeed, notice that we can obtain the rough estimate 
  \[ \label{eq:bad_Lipschitz_estimate_lambda}
  \big|\lambda^k(t_{h+1}) - \lambda^k(t_h) \big| \leq \frac{\sqrt{2}}{N_k} \, , \quad \text{almost surely.}  
  \]
  From this, we deduce that for every Lipschitz continuous function~$\psi$ with Lipschitz constant $\leq1$, we have that
\[
\begin{split}
  \int_{\Delta^{M-1}} \psi(\lambda) \, \d \big( \Lambda^k_{t_{h+1}} - \Lambda^k_{t_{h}})(\lambda) & = \int_{\Omega} \psi(\lambda^k(t_{h+1})) - \psi(\lambda^k(t_h)) \, \d \P  = \E\Big[ \psi(\lambda^k(t_{h+1})) - \psi(\lambda^k(t_h)) \Big] \\
  & \leq \E\Big[ \big| \lambda^k(t_{h+1}) - \lambda^k(t_h) \big| \Big] \leq \frac{\sqrt{2}}{N_k} = \sqrt{2} \tau_k^\alpha \, .
\end{split}
  \]
 Owing to Kantorovich's duality formula, we obtain that
 \[
\W_1(\Lambda^k_{t_{h+1}}, \Lambda^k_{t_h}) \leq \sqrt{2} \tau_k^\alpha \, .
\]
It seems that this estimate is not enough to deduce equicontinuity. 
Note that, iterating it yields
\[
\W_1(\Lambda^k_{t}, \Lambda^k_{s}) \leq C  \frac{|t-s|}{\tau_k } \frac{1}{N_k} \leq C |t-s| \frac{1}{\tau_k^{1-\alpha}}  \, ,
\]
an evidently bad estimate for equicontinuity purposes.

In the proof of the next theorem we avoid using~\eqref{eq:bad_Lipschitz_estimate_lambda}, replacing it with~\eqref{eq:compactness_estimate_final_final} below. 
However, this requires to provide a proof in the spirit of Arzelà--Ascoli's compactness result from scratch.
\end{remark}

\begin{theorem} \label{thm:properties_of_Lambda}
  Let $N_k \sim \tau_k^{-\alpha}$, $w_k = \tau_k^\beta$ and assume that $\alpha, \beta > 0$ and $\alpha + \beta = 1$, $\alpha > \frac{1}{2}$.
  Let $\Lambda^k_t$ be as in~\eqref{eq:Lambda_k_t} and $\overline \Lambda^k_t$ be as in~\eqref{eq:Lambda_bar_k_t}.  
  Then there exists a subsequence independent of $t$ (that we do not relabel) and a Lipschitz curve $t \in [0,T] \mapsto \Lambda_t \in \Prob_1(\Delta^{M-1})$ such that $\W_1(\Lambda^k_t,\Lambda_t) \to 0$ as $k \to +\infty$ for every $t \in [0,T]$.
\end{theorem}
\begin{proof}
  Before proving the compactness result, we need to establish a crucial estimate, \ie, \eqref{eq:compactness_estimate_final_final} below.
  We begin with an estimate for $s \leq t$ satisfying $s,t \in [t_h,t_{h+1}]$ for some $h \in \{0, \dots, k-1\}$.
  Let us fix $\psi \in C^\infty_c(\R^M)$.
  By the definition of $\Lambda^k$, we get that
  \[  \label{eq:integral_against_psi}
  \begin{split}
    & \int_{\Delta^{M-1}} \psi(\lambda) \, \d \Lambda^k_t(\lambda) - \int_{\Delta^{M-1}} \psi(\lambda) \, \d \Lambda^k_s(\lambda) = \int_{\Delta^{M-1}} \psi(\lambda) \, \d \big( \Lambda^k_t - \Lambda^k_s \big)(\lambda) \\
    & \quad = \int_{\Delta^{M-1}} \psi(\lambda) \, \d \big( \lambda^k(t) - \lambda^k(s) \big)_\# \P(\lambda) = \int_\Omega \big( \psi(\lambda^k(t)(\omega)) - \psi(\lambda^k(s)(\omega)) \big) \, \d \P(\omega) \\
    & \quad  = \E\big[ \psi(\lambda^k(t)) - \psi(\lambda^k(s)) \big]     \, .
  \end{split}
  \] 
  By Taylor's formula, we have that
  \[ \label{eq:taylor_formula_compactness}
  \begin{split}
    \psi(\lambda^k(t)) - \psi(\lambda^k(t_h)) & = \D \psi(\lambda^k(t_h)) \big[ \lambda^k(t) - \lambda^k(t_h)\big] + R(\lambda^k(t_h),\lambda^k(t)) \, , \\
    \psi(\lambda^k(s)) - \psi(\lambda^k(t_h)) & = \D \psi(\lambda^k(t_h)) \big[ \lambda^k(s) - \lambda^k(t_h)\big] + R(\lambda^k(t_h),\lambda^k(s)) \, ,
  \end{split}
  \]
  where the remainder is given by
  \[
  R(\lambda^k(t_h),\lambda^k(t)) = \! \int_0^1 (1-r) \D^2 \psi\big( \lambda^k(t_h) + r (\lambda^k(t)  - \lambda^k(t_h))\big) \big[ \lambda^k(t) - \lambda^k(t_h) ,  \lambda^k(t)  - \lambda^k(t_h) \big] \, \d r \, ,
  \]
  with the same formula with $s$ in place of $t$ for $R(\lambda^k(t_h),\lambda^k(s))$.
  Subtraction of the two equations in~\eqref{eq:taylor_formula_compactness} and~\eqref{eq:piecewise_affine_interpolation} yields
  \[   \label{eq:psi_t_psi_s}
  \begin{split}
    \psi(\lambda^k(t)) - \psi(\lambda^k(s)) & = \D \psi(\lambda^k(t_h)) \big[ \lambda^k(t) - \lambda^k(s)\big] + R(\lambda^k(t_h),\lambda^k(t)) - R(\lambda^k(t_h),\lambda^k(s)) \\ 
    & = \frac{t-s}{\tau_k} \D \psi(\lambda^k(t_h)) \big[   \lambda^k(t_{h + 1}) - \lambda^k(t_h)  \big] + R(\lambda^k(t_h),\lambda^k(t)) - R(\lambda^k(t_h),\lambda^k(s)) \, .
  \end{split}
  \]

  Let us estimate the expectation of $R(\lambda^k(t_h),\lambda^k(t))$ and $R(\lambda^k(t_h),\lambda^k(s))$. 
  We show the estimate for $R(\lambda^k(t_h),\lambda^k(t))$; the estimate for $R(\lambda^k(t_h),\lambda^k(s))$ is completely analogous.
  By the law of total probability, we have that
  \[ \label{eq:expectation_remainder_law_total_probability}
  \E\big[ R(\lambda^k(t_h),\lambda^k(t)) \big] = \sum_{\hat S} \E\big[ R(\hat \lambda, \lambda^k(t)) \ \big|\ S(t_h) = \hat S \big] \P\big(S(t_h) = \hat S\big) \, .
  \]
  Then, using that $\lambda^k(t)  - \hat \lambda = \frac{t-t_{h}}{\tau_k} (\lambda^k(t_{h+1}) - \hat \lambda)$, we estimate 
  \[ \label{eq:expectation_remainder}
  \begin{split}
    & \E\big[ R(\hat \lambda, \lambda^k(t)) \ \big|\ S(t_h) = \hat S \big]\\
    & \quad  = \E\Big[  \int_0^1 (1-s) \D^2 \psi\big( \hat \lambda + s (\lambda^k(t)  - \hat \lambda)\big) \big[ \lambda^k(t) - \hat \lambda ,  \lambda^k(t)  - \hat \lambda \big] \, \d s \ \Big|\ S(t_h) = \hat S \Big] \\
    & \quad = \sum_{i=1}^M \sum_{j=1}^M \int_0^1 (1-r) \D^2 \psi\Big( \hat \lambda + r \frac{t-t_h}{\tau_k}\frac{1}{N_k} (e_i - e_j) \Big) \Big[ \frac{t-t_h}{\tau_k} \frac{1}{N_k} (e_i - e_j) ,  \frac{t-t_h}{\tau_k} \frac{1}{N_k} (e_i - e_j) \Big] \x \\ 
    & \hspace{15em}\x \P\Big( \lambda^k(t_{h+1}) = \hat \lambda + \frac{1}{N_k}(e_i - e_j) \ \big|\ S(t_h) = \hat S \Big) \, \d r \\
    & \quad \leq \sum_{i=1}^M \sum_{j=1}^M \| \D^2 \psi\|_{L^\infty} \frac{2}{N_k^2} \frac{1}{\overline f(\hat \lambda)} \hat \lambda_{u_i} f_{u_i}(\hat \lambda) \hat \lambda_{u_j} = \frac{2}{N_k^2} \| \D^2 \psi\|_{L^\infty} \, .
  \end{split}
  \]
  Inserting this estimate in~\eqref{eq:expectation_remainder_law_total_probability}, we obtain that 
  \[ \label{eq:remainder}
    \E\big[ R(\lambda^k(t_h),\lambda^k(t)) \big]   \leq \frac{2}{N_k^2} \| \D^2 \psi\|_{L^\infty} = \O\Big(\frac{\|\D^2 \psi\|_{L^\infty}}{N_k^2}\Big) \, .
  \]

  We now exploit the estimate obtained in Lemma~\ref{lem:E_psi}.
  Taking the expectation in~\eqref{eq:psi_t_psi_s} and using~\eqref{eq:first_term} and~\eqref{eq:remainder}, we infer that
  \[ \label{eq:compactness_estimate}
  \begin{split}
    \E\big[ \psi(\lambda^k(t)) - \psi(\lambda^k(s)) \big] 
    & = \frac{t-s}{\tau_k} \Big( \frac{w_k}{N_k} \int_{\Delta^{M-1}} \D \psi(\lambda) b(\lambda) \, \d \Lambda^k_{t_h}(\lambda) \\
    & \quad + \frac{w_k}{N_k} \O\Big(\frac{\|\D\psi\|_{L^\infty}}{N_k} \Big) + \frac{w_k}{N_k} \O\big(\|\D\psi\|_{L^\infty} w_k \big) \Big) + \O\Big(\frac{\|\D^2 \psi\|_{L^\infty}}{N_k^2} \Big) \\ 
    & \quad \leq C |t-s| \frac{w_k}{N_k \tau_k} \Big( \|b\|_{L^\infty} + \frac{1}{N_k}   + w_k \Big)  \|\D\psi\|_{L^\infty}   +  C \frac{\|\D^2 \psi\|_{L^\infty}}{N_k^2}   \, .
  \end{split}
  \]

  Let us now fix generic $s \leq t$ in $[0,T]$. 
  Let us denote by $h' \in \{0, \dots, k-1\}$ the index such that $s \in [t_{h'},t_{{h'}+1})$ and by $h'' \in \{h', \dots, k-1\}$ the index such that $t \in [t_{h''},t_{h''+1})$ (if $t = T$, we set $h'' = k-1$).
  We estimate the difference $\E\big[\psi(\lambda^k(t)) - \lambda^k(s)\big]$ by concatenating the local inequality in~\eqref{eq:compactness_estimate} in the intervals $[t_{h'},t_{h'+1}), \ldots, [t_{h''},t_{h''+1})$:
  \[ 
  \begin{split}
    & \E\big[\psi(\lambda^k(t)) - \psi(\lambda^k(s))\big] \\
    & \quad  \leq \E\big[\psi(\lambda^k(t)) - \psi(\lambda^k(t_{h''}))\big] + \sum_{h=h'+1}^{h''-1} \E\big[\psi(\lambda^k(t_{h+1})) - \lambda^k(t_h)\big] + \E\big[\psi(\lambda^k(t_{h'+1})) - \psi(\lambda^k(s))\big] \\
    & \quad \leq  C |t-t_{h''}| \frac{w_k}{N_k \tau_k} \Big(\|b\|_{L^\infty} +  \frac{1}{N_k}   + w_k \Big)  \|\D\psi\|_{L^\infty} +  C \frac{\|\D^2 \psi\|_{L^\infty}}{N_k^2}  \\ 
    & \quad \quad + \sum_{h=h'+1}^{h''-1} \Big( C |t_{h+1}-t_h| \frac{w_k}{N_k \tau_k} \Big( \|b\|_{L^\infty} + \frac{1}{N_k}   + w_k \Big)  \|\D\psi\|_{L^\infty} +  C \frac{\|\D^2 \psi\|_{L^\infty}}{N_k^2} \Big) \\ 
    & \quad \quad + C |t_{h'+1}-s| \frac{w_k}{N_k \tau_k} \Big( \|b\|_{L^\infty} + \frac{1}{N_k}   + w_k \Big)  \|\D\psi\|_{L^\infty} +  C \frac{\|\D^2 \psi\|_{L^\infty}}{N_k^2}  \\
    & \quad \leq  C |t-s| \frac{w_k}{N_k \tau_k} \Big( \|b\|_{L^\infty} + \frac{1}{N_k}   + w_k \Big)  \|\D\psi\|_{L^\infty} +  C (h''-h'+1) \frac{\|\D^2 \psi\|_{L^\infty}}{N_k^2}  \, .
  \end{split}
  \]
  We remark that  
  \[
    h'' \tau_k - h' \tau_k = t_{h''} - t_{h'} \leq t-s + \tau_k \implies h'' - h' + 1 \leq \frac{|t-s|}{\tau_k} + 2 \leq \frac{C}{\tau_k}\, .
  \]
  It follows that
  \[ \label{eq:compactness_estimate_final}
  \E\big[\psi(\lambda^k(t)) - \psi(\lambda^k(s))\big] \leq C |t-s| \frac{w_k}{N_k \tau_k} \Big( \|b\|_{L^\infty} + \frac{1}{N_k}   + w_k \Big)  \|\D\psi\|_{L^\infty} +  C \frac{\|\D^2 \psi\|_{L^\infty}}{\tau_k N_k^2}  \, .
  \]
  Inserting this estimate in~\eqref{eq:integral_against_psi}, we have shown that
  \[ \label{eq:compactness_estimate_final_final}
  \begin{split}
    & \int_{\Delta^{M-1}} \psi(\lambda) \, \d \Lambda^k_t(\lambda) - \int_{\Delta^{M-1}} \psi(\lambda) \, \d \Lambda^k_s(\lambda) \\
    & \quad  \leq C |t-s| \frac{w_k}{N_k \tau_k} \Big( \|b\|_{L^\infty} + \frac{1}{N_k}   + w_k \Big)  \|\D\psi\|_{L^\infty} +  C \frac{\|\D^2 \psi\|_{L^\infty}}{\tau_k N_k^2}  \, ,
  \end{split}
  \]
  for every $s,t \in [0,T]$ and every $\psi \in C^\infty_c(\R^M)$.
 
  We are now in a position to prove the compactness result. 
  Let us fix a countable dense set $D \subset [0,T]$. 
  We exploit the compactness of $(\Prob_1(\Delta^{M-1}), \W_1)$ in Remark~\ref{rem:compactness_preliminaries}--\ref{item:P1_compact} to extract, by a diagonal argument, a subsequence independent of $q \in D$ (that we do not relabel) such that $\Lambda^k_q$ converges in $\W_1$ to some $\Lambda_q \in \Prob_1(\Delta^{M-1})$ for every $q \in D$.
  We now extend the convergence to the whole interval $[0,T]$.
  Let us fix $q, q' \in D$.
  Recalling the assumptions on $N_k$ and $w_k$, we have that 
  \[
  \frac{w_k}{N_k \tau_k}  \sim  \frac{\tau_k^\beta}{\tau_k \tau_k^{-\alpha}} = \tau_k^{\alpha + \beta - 1} = 1 \, , \quad \frac{1}{\tau_k N_k^2} \sim \frac{1}{\tau_k \tau_k^{-2\alpha}} = \tau_k^{2 \alpha - 1} \to 0 \, .
  \]
  Then we use the convergences $\W_1(\Lambda^k_q,\Lambda_q) \to 0$ and $\W_1(\Lambda^k_{q'},\Lambda_{q'}) \to 0$ as $k \to +\infty$ to pass to the limit in~\eqref{eq:compactness_estimate_final_final} with $s = q$ and $t = q'$ and obtain that
  \[  \label{eq:limit_estimate_on_D}
  \int_{\Delta^{M-1}} \psi(\lambda) \, \d \Lambda_q(\lambda) - \int_{\Delta^{M-1}} \psi(\lambda) \, \d \Lambda_{q'}(\lambda) \leq C |q-q'| \|b\|_{L^\infty} \|\D\psi\|_{L^\infty} \, ,
  \]
  for every $\psi \in C^\infty_c(\R^M)$. 
  
  We can generalize~\eqref{eq:limit_estimate_on_D} to 1-Lipschitz competitors $\psi$ by an approximation argument. 
  Indeed, let $\psi \colon \Delta^{M-1} \to \R$ be a 1-Lipschitz function.
  By Kirszbraun's extension theorem, we can extend $\psi$ to a 1-Lipschitz function (not relabeled) $\psi \colon \R^M \to \R$.
  We can then approximate $\psi$ with a family of smooth functions $\psi_\e \in C^\infty_c(\R^M)$ obtained by convolution with a mollifier and a multiplication by a compactly supported cut-off function identically equal to 1 on a neighborhood of $\Delta^{M-1}$.
  Since $\psi_\e \to \psi$ locally uniformly, we pass to the limit in~\eqref{eq:compactness_estimate_final_final} with $\psi_\e$ in place of $\psi$ and obtain that~\eqref{eq:limit_estimate_on_D} holds for every 1-Lipschitz function $\psi$.

Taking the supremum in~\eqref{eq:limit_estimate_on_D} over all 1-Lipschitz functions $\psi$, we obtain that 
\[ \label{eq:Lipschitz_estimate_on_D}
\W_1(\Lambda_q,\Lambda_{q'}) \leq C |q-q'| \, .
\]
This implies that for $q \to t$ and $q' \to t$, with $q,q' \in D$, we have that $\W_1(\Lambda_q,\Lambda_{q'}) \to 0$, \ie, $\{\Lambda_q \ | \ q \to t\, , q \in D\}$ is a Cauchy sequence in $(\Prob_1(\Delta^{M-1}), \W_1)$.
Since $(\Prob_1(\Delta^{M-1}), \W_1)$ is a complete metric space, we deduce that there exists a limit $\Lambda_t \in \Prob_1(\Delta^{M-1})$ such that $\W_1(\Lambda_q,\Lambda_t) \to 0$ as $q \to t$.

We observe that $\W_1(\Lambda^k_t,\Lambda_t) \to 0$ as $k \to +\infty$ for every $t \in [0,T]$.
For, given $q \to t$, $q \in D$, we have that
\[
\W_1(\Lambda^k_t,\Lambda_t) \leq \W_1(\Lambda^k_t,\Lambda^k_q) + \W_1(\Lambda^k_q,\Lambda_q) + \W_1(\Lambda_q,\Lambda_t)  \, ,
\] 
and the right-hand side converges to 0 as $k \to +\infty$ and then $q \to t$.
  
Finally, passing to the limit in~\eqref{eq:Lipschitz_estimate_on_D} for $q \to t$ and $q' \to s$, we obtain that
\[
\W_1(\Lambda_t,\Lambda_s) \leq C |t-s| \, ,
\]
for every $s,t \in [0,T]$, \ie, $\Lambda \in C([0,T];\Prob_1(\Delta^{M-1}))$ is Lipschitz continuous.

This concludes the proof of the proposition.
\end{proof}
 
We are ready to prove the main result of this paper.

\begin{theorem}
  Let $N_k \sim \tau_k^{-\alpha}$, $w_k = \tau_k^\beta$ and assume that $\alpha, \beta > 0$ and $\alpha + \beta = 1$, $\alpha > \frac{1}{2}$. 
  Let $\Lambda^k_t$ be as in~\eqref{eq:Lambda_k_t} and $\overline \Lambda^k_t$ be as in~\eqref{eq:Lambda_bar_k_t}.  
  Let $t \in [0,T] \mapsto \Lambda_t \in \Prob_1(\Delta^{M-1})$ be the Lipschitz curve obtained in Theorem~\ref{thm:properties_of_Lambda} (we do not relabel the subsequence).
  Then $\Lambda_t$ is a distributional solution to 
  \[ \label{eq:continuous_time_limit}
    \de_t \Lambda_t + \div \big(  b  \Lambda_t \big) = 0 \, , \quad \text{in } (0,T) \times \Delta^{M-1} \, , 
  \]
  with initial condition $\Lambda_0$, \ie, 
  \[
  \int_0^T \int_{\Delta^{M-1}} \big( \de_t \phi(t,\lambda) + \D \phi(t,\lambda) b(\lambda)   \big) \, \d \Lambda_t(\lambda) \, \d t = - \int_{\Delta^{M-1}} \phi(0,\lambda) \, \d \Lambda_0(\lambda) \, ,
  \]
  for every $\phi \in C^\infty_c((-\infty,T) \times \R^M)$.
\end{theorem}
\begin{proof}
Let us fix $\phi \in C^\infty_c((-\infty,T) \times \R^M)$.
By the assumptions on $N_k$ and $w_k$ and by Proposition~\ref{prop:approximate_distributional_solution}, we have that
\[
\int_0^T \int_{\Delta^{M-1}} \de_t \phi(t,\lambda)  \, \d \Lambda^k_t(\lambda) \, \d t + \int_0^T \int_{\Delta^{M-1}} \D \phi(t,\lambda) b(\lambda)   \, \d \overline \Lambda^k_t(\lambda) \, \d t  = - \int_{\Delta^{M-1}} \phi(0,\lambda) \, \d \Lambda^k_0(\lambda) + \tilde \rho_k \, ,
\]
where
\[
\tilde \rho_k \sim \tau_k^{\alpha + \beta - 1} \O\Big( \frac{\|\D\psi\|_\infty}{N_k} \Big) + \tau_k^{\alpha + \beta - 1} \O\big(\|\D \psi\|_\infty w_k \big)  +  \tau_k^{2\alpha - 1}\O\big(\|\D^2 \psi\|_{L^\infty} \big) \to 0 \, , \quad \text{as } k \to +\infty \, .
\]
Moreover, by Theorem~\ref{thm:properties_of_Lambda}, we have that 
\[
  \int_0^T \int_{\Delta^{M-1}} \de_t \phi(t,\lambda)  \, \d \Lambda^k_t(\lambda) \, \d t \to \int_0^T \int_{\Delta^{M-1}} \de_t \phi(t,\lambda)  \, \d \Lambda_t(\lambda) \, \d t \, , \quad \text{as } k \to +\infty \, ,
\]
and, exploiting also Lemma~\ref{lem:Lambda_k_Lambda_bar},
\[
  \int_0^T \int_{\Delta^{M-1}} \D \phi(t,\lambda) b(\lambda)   \, \d \overline \Lambda^k_t(\lambda) \, \d t \to \int_0^T \int_{\Delta^{M-1}} \D \phi(t,\lambda) b(\lambda)   \, \d \Lambda_t(\lambda) \, \d t \, , \quad \text{as } k \to +\infty \, .
\]
We have shown that 
\[
\int_0^T \int_{\Delta^{M-1}} \de_t \phi(t,\lambda)  \, \d \Lambda_t(\lambda) \, \d t + \int_0^T \int_{\Delta^{M-1}} \D \phi(t,\lambda) b(\lambda)   \, \d \Lambda_t(\lambda) \, \d t  = - \int_{\Delta^{M-1}} \phi(0,\lambda) \, \d \Lambda_0(\lambda) \, ,
\]
concluding the proof of the theorem.
\end{proof}

\begin{remark}
  In the next section we recall that the continuity equation~\eqref{eq:continuous_time_limit} has a unique solution. 
  This implies that the limit curve $t \in [0,T] \mapsto \Lambda_t \in \Prob_1(\Delta^{M-1})$ is independent of the subsequence extracted in Theorem~\ref{thm:properties_of_Lambda}.
  Hence, \emph{a posteriori}, the convergence $\W_1(\Lambda^k_t,\Lambda_t) \to 0$ as $k \to +\infty$ holds for the whole sequence $k \to +\infty$.
\end{remark}
  
\section{Replicator dynamics} \label{sec:replicator_dynamics}

We conclude this section by recalling the connection between the continuity equation 
\[
\de_t \Lambda_t + \div \big(  b  \Lambda_t \big) = 0 \, , \quad \text{in } (0,T) \times \Delta^{M-1} \, ,
\]
and the replicator dynamics.
This will explain in which sense the replicator dynamics is a mean-field limit of the finite-population stochastic process described in Section~\ref{sec:discrete_stochastic_process}.

The replicator dynamics is governed by the following system of ODEs:
\[
\begin{cases}
  \displaystyle \frac{\d}{\d t} \lambda_{u_i}(t) = \lambda_{u_i}(t) \big( (A \lambda(t))_i - (A \lambda(t)) \cdot \lambda(t) \big) \, , &  i = 1, \dots, M \, , \ t \in (0,T) \, , \\
  \lambda(0) = \lambda_0 \, .
\end{cases}
\]
In a more compact form, using the vector field $b \colon \Delta^{M-1} \to \R^M$ introduced in~\eqref{eq:b}, 
\[
\begin{cases}
\displaystyle \frac{\d}{\d t} \lambda(t) = b(\lambda(t)) \, , &  t \in (0,T) \, , \\
\lambda(0) = \lambda_0 \, .
\end{cases}
\]
Let us denote by $\Psi \colon (0,T) \times \Delta^{M-1} \to \Delta^{M-1}$ the flow of the vector field $b$, \ie, $\Psi(t,\lambda_0)$ is the solution of the ODE system above with initial condition $\lambda_0$.

A first connection between the continuity equation and the replicator dynamics is given by the following result.
This fact is known, see, \eg, \cite[Lemma~8.1.6]{AmbGigSav08}.

\begin{proposition}
   Let $\Lambda_0 \in \Prob_1(\Delta^{M-1})$ and let 
  \[ \label{eq:transported_measure}
  \Lambda_t := \Psi(t, \cdot)_\# \Lambda_0 \in \Prob_1(\Delta^{M-1}) \, , \quad t \in [0,T) \, .
  \]
  Then $\Lambda_t$ is a distributional solution to the continuity equation
  \[
  \de_t \Lambda_t + \div \big(  b  \Lambda_t \big) = 0 \, , \quad \text{in } (0,T) \times \Delta^{M-1} \, ,
  \]
  with initial condition $\Lambda_0$.
\end{proposition}

In fact, \eqref{eq:transported_measure} characterizes the solution of the continuity equation. 
See also~\cite[Proposition~8.1.8]{AmbGigSav08}.

\begin{theorem}
  Let $\Lambda_0 \in \Prob_1(\Delta^{M-1})$. Let $t \in [0,T] \mapsto \Lambda_t \in \Prob_1(\Delta^{M-1})$ be a narrowly continuous path, being a distributional solution to the continuity equation  
  \[
    \de_t \Lambda_t + \div \big(  b  \Lambda_t \big) = 0 \, , \quad \text{in } (0,T) \times \Delta^{M-1} \, .
  \]
  Then 
  \[
  \Lambda_t = \Psi(t, \cdot)_\# \Lambda_0 \, , \quad t \in [0,T) \, .
  \]
\end{theorem}

The two results above explain the connection between the continuity equation and the replicator dynamics.

\vspace{2em}

{\bfseries Acknowledgements.} M.\ Morandotti and G.\ Orlando are members of Gruppo Nazionale per l'Analisi Matematica, la Probabilit\`a e le loro Applicazioni (GNAMPA) of the Istituto Nazionale di Alta Matematica (INdAM).

  M.\ Morandotti has been partially supported by the GNAMPA project 2024 \emph{Analisi asintotica di modelli evolutivi di interazione} from Istituto Nazionale di Alta Matematica ``F.\ Severi''.
  
  M.\ Morandotti's research was partly carried out within the project \emph{Geometric-Analytic Methods for PDEs and Applications} funded by European Union - Next Generation EU within the MUR PRIN 2022 program (D.D.\ 104 - 02/02/2022 - project code 2022SLTHCE). This manuscript reflects only the authors’ views and opinions and the Ministry cannot be considered responsible for them.

  G.\ Orlando has been partially supported by the Research Project of National Relevance ``Evolution problems involving interacting scales'' granted by the Italian Ministry of Education, University and Research (MUR PRIN 2022, project code 2022M9BKBC, Grant No. CUP D53D23005880006).

  G.\ Orlando acknowledges financial support under the National Recovery and Resilience Plan (NRRP) funded by the European Union - NextGenerationEU -  
Project Title ``Mathematical Modeling of Biodiversity in the Mediterranean sea: from bacteria to predators, from meadows to currents'' - project code P202254HT8 - CUP B53D23027760001 -  
Grant Assignment Decree No. 1379 adopted on 01/09/2023 by the Italian Ministry of University and Research (MUR).

G.\ Orlando was partially supported by the Italian Ministry of University and Research under the Programme ``Department of Excellence'' Legge 232/2016 (Grant No. CUP - D93C23000100001).

\bibliographystyle{siam}  
\bibliography{bibliography}

\begin{thebibliography}{10}

\bibitem{AlmMorSol21}
{\sc S.~Almi, M.~Morandotti, and F.~Solombrino}, {\em A multi-step lagrangian
  scheme for spatially inhomogeneous evolutionary games}, J. Evol. Equ., 21
  (2021), pp.~2691--2733.

\bibitem{AmbForMorSav21}
{\sc L.~Ambrosio, M.~Fornasier, M.~Morandotti, and G.~Savaré}, {\em Spatially
  inhomogeneous evolutionary games}, Comm. Pure Appl. Math., LXXIV (2021),
  pp.~1353--1402.

\bibitem{AmbGigSav08}
{\sc L.~Ambrosio, N.~Gigli, and G.~Savar\'e}, {\em Gradient flows in metric
  spaces and in the space of probability measures}, Lectures in Mathematics ETH
  Z\"urich, Birkh\"auser Verlag, Basel, second~ed., 2008.

\bibitem{AmbTre14}
{\sc L.~Ambrosio and D.~Trevisan}, {\em Well-posedness of {L}agrangian flows
  and continuity equations in metric measure spaces}, Anal. PDE, 7 (2014),
  pp.~1179--1234.

\bibitem{BorSar97}
{\sc T.~B\"orgers and R.~Sarin}, {\em Learning through reinforcement and
  replicator dynamics}, Journal of Economic Theory, 77 (1997), pp.~1--14.

\bibitem{Car21}
{\sc R.~Carmona}, {\em Applications of mean field games in financial
  engineering and economic theory}, in Mean Field Games. Agent Based Models to
  Nash Equilibria, F.~Delarue, ed., AMS, 2021, pp.~165--218.

\bibitem{ChaSou09}
{\sc F.~A. C.~C. Chalub and M.~O. Souza}, {\em From discrete to continuous
  evolution models: A unifying approach to drift-diffusion and replicator
  dynamics}, Theoretical Population Biology, 76 (2009), pp.~268--277.

\bibitem{ChaSou14}
\leavevmode\vrule height 2pt depth -1.6pt width 23pt, {\em The
  frequency-dependent {W}right--{F}isher model: diffusive and non-diffusive
  approximations}, J. Math. Biol., 68 (2014), pp.~1089--1133.

\bibitem{DorBlu05}
{\sc M.~Dorigo and C.~Blum}, {\em Ant colony optimization theory: a survey},
  Theoret. Comput. Sci., 344 (2005), pp.~243--278.

\bibitem{DMPW09}
{\sc B.~D\"uring, P.~Markowich, J.-F. Pietschmann, and M.-T. Wolfram}, {\em
  Boltzmann and {F}okker-{P}lanck equations modelling opinion formation in the
  presence of strong leaders}, Proc. R. Soc. Lond. Ser. A Math. Phys. Eng.
  Sci., 465 (2009), pp.~3687--3708.

\bibitem{Hil11}
{\sc C.~Hilbe}, {\em Local replicator dynamics: A simple link between
  deterministic and stochastic models of evolutionary game theory}, Bull Math
  Biol, 73 (2011), pp.~2068--2087.

\bibitem{HofSig98}
{\sc J.~Hofbauer and K.~Sigmund}, {\em Evolutionary Games and Population
  Dynamics}, Cambridge University Press, Cambridge, 1998.

\bibitem{Hol75}
{\sc J.~H. Holland}, {\em Adaptation in natural and artificial systems: an
  introductory analysis with applications to biology, control, and artificial
  intelligence}, University of Michigan Press, Ann Arbor, 1975.

\bibitem{Ken10}
{\sc J.~Kennedy}, {\em Particle swarm optimization}, in Encyclopedia of Machine
  Learning, C.~Sammut and G.~I. Webb, eds., Springer US, Boston, MA, 2010,
  pp.~760--766.

\bibitem{Mor62}
{\sc P.~A.~P. Moran}, {\em The Statistical Process of Evolutionary Theory},
  Clarendon Press, Oxford, 1962.

\bibitem{MorSol20}
{\sc M.~Morandotti and F.~Solombrino}, {\em Mean-field analysis of
  multipopulation dynamics with label switching}, SIAM J. Math. Anal., 52
  (2020), pp.~1427--1462.

\bibitem{Now06}
{\sc M.~A. Nowak}, {\em Evolutionary Dynamics: Exploring the Equations of
  Life}, Belknap Press, 2006.

\bibitem{SmiPri73}
{\sc J.~M. Smith and G.~R. Price}, {\em The logic of animal conflict}, Nature,
  246 (1973), pp.~15--18.

\bibitem{TayJon78}
{\sc P.~D. Taylor and L.~B. Jonker}, {\em Evolutionarily stable strategies and
  game dynamics}, Math. Biosci., 40 (1978), pp.~145--156.

\bibitem{Tos06}
{\sc G.~Toscani}, {\em Kinetic models of opinion formation}, Commun. Math.
  Sci., 4 (2006), pp.~481--496.

\bibitem{TraClaHau05}
{\sc A.~Traulsen, J.~C. Claussen, and C.~Hauert}, {\em Coevolutionary dynamics:
  From finite to infinite populations}, Phys. Rev. Lett., 95 (2005), p.~238701.

\bibitem{TraClaHau06}
\leavevmode\vrule height 2pt depth -1.6pt width 23pt, {\em Coevolutionary
  dynamics in large, but finite populations}, Phys. Rev. E, 74 (2006),
  p.~011901.

\bibitem{Vil08}
{\sc C.~Villani}, {\em Optimal Transport: Old and New}, Springer, Berlin, 2008.

\end{thebibliography}
\end{document}